\documentclass[12pt]{article}
\usepackage{amsmath,amssymb,amscd,array,mathrsfs,mathtools,bm,stmaryrd} % Ensure proper font support
\usepackage{soul}
\usepackage{titlesec}
\usepackage{shuffle}
\usepackage{amsmath,amscd,amssymb,amsthm,array}
\usepackage{color}
\allowdisplaybreaks
\usepackage{amsmath,amssymb,mathrsfs,amsthm, tikz-cd,mathrsfs}
\usepackage{comment}
\def\url#1{\expandafter\s

\tring\csname #1\endcsname}

\def\mmat #1,#2,#3,#4,{\text{\small\arraycolsep=3pt $
\begin{pmatrix}#1&#2\\#3&#4\end{pmatrix}$}}

\usepackage{hyperref}
\usepackage{capt-of}

\usepackage{multirow}

\usepackage{comments}
\newComments\NG{Nico}{red}
\newComments\QEh{QEh}{blue}

\def\mmat #1,#2,#3,#4,{\text{\small\arraycolsep=3pt $
\begin{pmatrix}#1&#2\\#3&#4\end{pmatrix}$}}

\usepackage{lscape}
\usepackage{tikz-cd}
\usepackage{enumerate}

\usepackage{DLdef1}
\usepackage{multicol}

\def\mmat #1,#2,#3,#4,{\text{\small\arraycolsep=3pt $
\begin{pmatrix}#1&#2\\#3&#4\end{pmatrix}$}}

\renewcommand {\ssbegin}[2][*]
 {\refstepcounter{subsection}%
\if#1*
\addcontentsline{toc}{subsection}{\thesubsection.\hskip 1pc #2}%
\else
\addcontentsline{toc}{subsection}{\thesubsection.\hskip 1pc #2. #1}%
\fi
 \def \secno {\gdef \secno {}{\ssecfont
\thesubsection.\hskip 2ex}%
 }%
 \begin{#2}}

\renewcommand {\sssbegin}[2][*]
 {\refstepcounter{subsubsection}%\label{sss#1}
\if#1*
\addcontentsline{toc}{subsubsection}{\thesubsubsection.\hskip 1pc #2}%
\else
\addcontentsline{toc}{subsubsection}{\thesubsubsection.\hskip 1pc #2. #1}
\fi
 \def \secno {\gdef \secno {}{\ssecfont \thesubsubsection.\hskip 2ex}%
 }%
 \begin{#2}}

\renewcommand {\parbegin}[2][*]
 {\refstepcounter{paragraph}%\label{ssss#1}
\if#1*
\addcontentsline{toc}{paragraph}{\theparagraph.\hskip 1pc #2}%
\else
\addcontentsline{toc}{paragraph}{\theparagraph.\hskip 1pc #2. #1}
\fi
 \def \secno {\gdef \secno {}{\ssecfont \theparagraph.\hskip 2ex}%
 }%
 \begin{#2}}

\newcommand{\tht}{\theta}

\rmnameii{etr}{etr}
\rmnameii{evv}{ev}

\newcommand {\A}{{\cal{A}}}
\newcommand{\D}{{\cal{D}}}

\newcommand{\btr}{\blacktriangleright}
\newcommand{\mo}{\mathcal{O}}

\pgfdeclarelayer{nodelayer}
\pgfdeclarelayer{edgelayer}
\pgfsetlayers{edgelayer,nodelayer,main}
\tikzstyle{noneq}=[circle,fill=black,draw=black]

\newtheorem{theorem}{Theorem}[subsection]
\newtheorem{example}[theorem]{Example}
\newtheorem{proposition}[theorem]{Proposition}
\newtheorem*{proposition*}{Proposition}
\newtheorem{corollary}[theorem]{Corollary}
\newtheorem*{theorem*}{Theorem}
\newtheorem{lemma}[theorem]{Lemma}
\newtheorem{remark}[theorem]{Remark}
\newtheorem{definition}[theorem]{Definition}

\date{}

\titleformat*{\section}{\large\bfseries}
\titleformat*{\subsection}{\large\bfseries}
\titleformat*{\subsubsection}{\normalsize\bfseries}

\begin{document}

\title{\Large \textsc{Jacobson identities for post-Lie algebras \\ in positive characteristic}}

\author{\normalsize Quentin Ehret, Nicolas Gilliers {}\footnote{Division of Science and Mathematics, New York University Abu Dhabi, P.O. Box 129188, Abu Dhabi, United Arab Emirates. Email: \texttt{qe209@nyu.edu,nag9000@nyu.edu}. QE is supported by the Grant NYUAD-065.}}

% \thanks{QE is supported by the grant NYUAD-065. NG is supported by the grant ...}

% \subjclass[2020]{\red??? \black}

%\subjclass[2010]{Primary 17B50; Secondary 17B20}

\maketitle

\begin{abstract}
Let $p$ be a prime number. Given a restricted Lie algebra over a field of characteristic $p$ and a post-Lie operation over it, we prove the Jacobson identities for a $p$-structure built from the Lie bracket and the post-Lie operation, called sub-adjacent $p$-structure. Furthermore, we give sufficient conditions for the sub-adjacent Lie algebra to be restricted if equipped with this sub-adjacent $p$-structure. This construction is ``axiomatized'' by introducing the notion of restricted post-Lie algebras, and we work out several examples. 
\end{abstract}
\thispagestyle{empty}
\setcounter{tocdepth}{2}
\tableofcontents

%%%%%%%%%%%%%%%%%%%%%%%%%%%%%%%%%%%%%%%%%%%%%%%%%%%%%%%%%%%%%%%%%%%%%%%

% \begin{enumerate}
% \item \hl{We need an example of a restricted postLie which is not trivially restricted}. \red je vais calculer sur heisenberg et p=3. Je vais essayer de coder un truc pour calculer directement étant donnée l'algèbre d elie restreitne de départ.\black
% \item \hl{Remark 2.1.4 in Dokas.} 
% \end{enumerate}

\section{Introduction} \label{intro}

This paper is devoted to the study of post-Lie algebras over a field of positive characteristic. The post-Lie operad was introduced in characteristic zero by Vallette in 2007 as the Koszul dual of the operad of commutative triassociative algebras in the context of algebraic combinatorics, see \cite{Va}. Post-Lie algebras independently appeared in the study of numerical integrators using non-commutative Lie groups, see \cite{MKW}. Since then, this notion found applications in various contexts, like the study of affine connections with constant torsion (see \cite{MKL}), non-commutative probability theory (see \cite{G}) and renormalisation and regularity structures (see e.g \cite{BK}), to cite a few. Post-Lie algebras also naturally generalize pre-Lie algebras, as a pre-Lie algebra can be seen as a post-Lie structure over an abelian Lie algebra (see Definition \ref{def:postLie}). In characteristic zero, pre-Lie algebras (also called right-symmetric algebras) were introduced by Gerstenhaber in his work on Hochschild cohomology and deformations of algebras (see \cite{Ge}), but made their first appearance in the mathematics literature way before in the context of rooted trees algebras; see \cite{C}. Over the years, this type of algebras has been instrumental in many areas of mathematics; they share, for instance, profound relations with the theory of covariant derivatives on flat spaces. Trivial examples of pre-Lie algebras include associative algebras. For survey papers on this notion, see \cite{Bu,M}. In \cite{OG}, a procedure known as the Guin–Oudom process is described for constructing certain universal enveloping algebras of pre-Lie algebras. Notably, this process is used to build Hopf algebras related to the Butcher–Connes–Kreimer Hopf algebra (see \cite{CK}), which has applications in renormalization and stochastic analysis, see e.g. \cite{BK,F}. In the present paper, we adapt and use this procedure for post-Lie algebras in positive characteristic.

Over a field of positive characteristic $p>0$, most of the techniques related to Lie algebras are no longer valid and new ideas are required. In fact, an additional structure appears on many Lie algebras, known as ``$p$-map" or ``Frobenius map", which behaves roughly like the $p$-th power of associative algebras. Lie algebras equipped with such structure are called \textit{restricted} and were introduced by Jacobson in the 1940's in the study of the derivation algebra of associative algebras, see \cite{J}. Restricted Lie algebras are of prime importance due to their link to algebraic groups and their role in representation theory, see \cite{SF}. In particular, with $A$ an associative algebra over a field of characteristic $p>0$, Jacobson proved that $(a+b)^{p} - a^p - b^p$ is a sum of Lie commutators in $A$, for all $a,b\in A$. This is called the \emph{Jacobson identity} and is axiomatized in the definition of restricted Lie algebras, see Equation \eqref{def:restrictedLie}. The adaptation of usual methods to pre-Lie algebras in positive characteristic also necessitates the introduction of additional structure. Restricted pre-Lie algebras were first introduced by Dzhumadil'daev in \cite{asqar}. He showed analogs of Jacobson identities for the $p$-th power of the (non-associative) pre-Lie product and demonstrated that the sub-adjacent Lie algebra obtained by taking the commutator of the pre-Lie product is in fact restricted, see \cite[Theorem 1.1]{asqar}).  Later, Dokas extended this concept by generalizing the notion of restricted pre-Lie algebras to accommodate an abstract $p^{th}$ power (see \cite[Definition 2.3]{Do}). Notably, he showed that dendriform algebras possess two restricted structures; one in Dzhumadil’daev’s sense and another in the broader sense he introduced.

The main goal of the present paper is to introduce the notion of \textit{restricted} post-Lie algebra as an analogue of the notions of restricted Lie and restricted pre-Lie algebras, building upon Dzhumadil’daev's and Dokas' works.

\subsection{Contributions.}
This paper introduces the two notions of \emph{trivially restricted post-Lie algebras} and \emph{restricted post-Lie algebras}. We use the Guin-Oudom construction, introduced after \cite{asqar} was published, to perform our calculations, which makes them simpler. Using this construction, we recover the results of \cite{asqar} on pre-Lie algebras and extend them to post-Lie algebras. 

Unless otherwise specified, we denote by $k$ a field of characteristic $p\geq 2$. Through the paper, ``ordinary" is understood as ``not restricted".

To be more precise, let $(\fg, [-,-], (-)^{[p]})$ be a restricted Lie algebra equipped with a post-Lie product $\btr$ (see Section \ref{sec:postLie}) over a field $k$ of characteristic $p>0$. Recall that a map $f:V\rightarrow W$ between $k$-vector spaces is called \textit{$p$-linear} if it satisfies
\begin{equation}
    f(\lambda x+y)=\lambda^pf(x)+f(y),~\forall x,y\in V,\, \forall\lambda \in k.
\end{equation}
Define a $p$-linear map (we shall prove it further below) $\fg \ni x \mapsto x^{[p]_{\btr}}\in \fg$ by, for all $x\in \fg$:
\begin{equation}
\label{eqn:pstruct}
x^{[p]_{\btr}} := x^{[p]} +\sum_{n=1}^{p-1}\frac{1}{n}\sum_{\substack{\ell_1 + \cdots + \ell_{n} = p \\ \ell_i \geq 1}} \frac{(p-1)!}{(\ell_1-1)!\cdots (\ell_n-1)!}\prod_{1 \leq v \leq n-1} ( \ell_{1} + \cdots + \ell_v)^{-1}\,[x^{\bullet\,\ell_1} \cdots x^{\bullet\,\ell_{n}}]
\end{equation}
For any tuple $ \ell_1,\ldots,\ell_n \geq 1$ of $n$ positive integers summing to $p$, let
\begin{align}
\label{eqn:coeffc}
C_{\ell_1,\ldots,\ell_k}= \prod_{1 \leq v \leq k-1} ( \ell_{1} + \cdots + \ell_v) \quad \text{mod }p.
\end{align}
The bracket $[\cdots]$ and $\bullet$ are defined by:
\begin{align}
 [x_1,\ldots,x_n] &= [x_1,[x_2,\cdots [x_{n-1},x_n]]],\quad && \forall x_1,\ldots,x_n \in \fg, \\
 x^{\bullet\, n} &= (n{\rm \,times}) \,\, x \btr ( x \btr ( \cdots (x\btr x)\cdots)), \quad && \forall x \in \fg.
\end{align}
For example, when $p=2$, $C_{1} = 1$ (the product is empty) and 
\begin{equation}x^{[2]_{\btr}} = x^{[2]}+x\btr x,\quad \forall x \in \fg.\end{equation}
When $p=3$, we find the following formulae
\begin{align}
&C_{10} = 2 = -1,\quad C_{01} = 1, \\ 
& x^{[3]_{\btr}} = x^{[3]} + x\btr(x \btr x) + [x\btr x,x],\quad \forall x\in \fg.
\end{align}
The specific case of $p=2$ (resp. $p=3$) is investigated in Section \ref{sec:p=2} (resp. \ref{sec:p=3}).

We will see that the coefficients appearing in formula \eqref{eqn:pstruct} have a nice interpretation being the number of linearizations of certain binary trees. Observe that, already when $p=2$, many cancellations occur in the sum on the right-hand side of \eqref{eqn:pstruct}, which we shall discuss in Section \ref{sec:simplifications} 

The first result of this paper, see Theorem \ref{prop:jacobsonidentitties}, states that the map $x \mapsto x^{[p]_{\btr}}$ satisfies the Jacobson identities relatively to the sub-adjacent Lie bracket to the post-Lie product $\btr$. Our second result expands $x^{[p]_{\btr}}$ over a certain family of right-iterated brackets, see Theorem \ref{thm:expansion}

 \subsection{Restricted Lie algebras.}\label{sec:restrictedLie}Let $\fg$ be a finite-dimensional Lie algebra over $k$. Following Jacobson (\cite{J}), a set-map 
    $$(-)^{[p]}:\fg\rightarrow \fg, \quad x\mapsto x^{[p]},
    $$ is called a~\textit{$p$-map} on $\fg$ and $\fg$ is said to be {\it restricted}  if 
\begin{align}
\label{def:restrictedLie}
&(\lambda x)^{[p]}=\lambda^p x^{[p]} \text{ for all } x\in \fg \text{ and for all } \lambda \in k; \nonumber\\
&\ad_{x^{[p]}}=(\ad_x)^p \text{ for all } x\in \fg, \nonumber \\
&(x+y)^{[p]}=x^{[p]}+y^{[p]}+\displaystyle\sum_{1\leq i\leq p-1}s_i(x,y), \text{ for all } x,y\in \fg 
\end{align}
where the coefficients $s_i(x,y)$ are obtained from the expansion
\begin{equation*}
(\ad_{\lambda x+y})^{p-1}(x)=\displaystyle\sum_{1\leq i \leq p-1} is_i(x,y) \lambda^{i-1}.
\end{equation*}
Equation \eqref{def:restrictedLie} is referred to as the Jacobson identities, which can be made more explicit as follows:
\begin{align}\label{si}
\sum_{i=1}^{p-1}s_i(x,y)&=\sum_{\underset{x_{p-1}=y,~x_p=x}{x_k\in\{x,y\}}}\frac{1}{\sharp(x)}[x_1,[x_2,[...,[x_{p-1},x_p]...]]],\quad \forall x,y\in \fg,
\end{align}
where $\sharp(x):=\text{card}\{k,~x_k=x,~k=1,\cdots,p\}$.\\

The following Theorem is due to Jacobson (see \cite{J}) and is useful to study $p$-maps on Lie algebras.
\begin{theorem}\label{thm:jacobson}
Let $(L,[-,-])$ be a Lie algebra and let $(e_j)_{j\in J}$ be a~basis of $L$ such that there are $f_j\in L$ satisfying $(\ad_{e_j})^p=\ad_{f_j}$. Then, there exists exactly one $p$-map  $(\cdot)^{[p]}:L\rightarrow L$ such that $e_j^{[p]}=f_j \quad \text{ for all $j\in J$}.$
\end{theorem}
 
\subsubsection{Examples.}
\label{examplesun}
\begin{enumerate}
\item[($i$)] Let $\fg$ be an abelian Lie algebra. Then, any $p$-linear map $\fg\rightarrow\fg$ yields a restricted structure on $\fg$.
\item[($ii$)] Let $\A$ be an associative algebra over $k$. Then $(\A,[-,-],(-)^{[p]})$ is a restricted Lie algebra with the bracket  $[a,b]:=ab-ba$ and the $p$-map $a^{[p]}:=a^p,~\forall a,b\in \A$  (see \cite[Section 2.1]{SF}).
\item[($iii$)] Recall that given an associative algebra $\A$, the derivations space of $\A$, denoted by $\Der(\A)$, is defined by
\begin{equation*}
    \Der(\A):=\{d:\A\rightarrow\A,~d(ab)=d(a)b+ad(b),~\forall a,b\in \A\}.
\end{equation*}
It is well known (see \cite{J}) that $\Der(\A)$ is a restricted Lie algebra, with
\begin{align*}
    [f,g]&:=f\circ g-g\circ f;\\
    f^{[p]}&:=f^p={f\circ f\circ \cdot  \cdot \circ f} \,\,(p ~{\rm times}).
\end{align*}
\item [($iv$)]
(Partial restricted version of \cite[Theorem 3.1]{JZ})\label{prop:tensorderivation}
    Let $\A$ be an  associative algebra. Let $\D\subset\Der(\A)$ be a restricted subalgebra. Then, the space $\A\otimes\D$ is a restricted Lie algebra with, for all $a,b\in \A$ and for all $f,g\in\D$,
    \begin{align}\label{eq:tensorbracket}
        [a\otimes f,b\otimes g]&:=ab\otimes[f,g]; \\
        (a\otimes f)^{[p]}&:=a^p\otimes f^p. \nonumber
    \end{align}
    The proof is straightforward. Let $a,b\in \A$ and $f,g\in\Der(\A)$.
    \begin{align*}
        \bigl[(a\otimes f)^{[p]},b\otimes g\bigl]=[a^p\otimes f^p,b\otimes g]=a^pb\otimes[f^p,g]=a^pb\otimes\ad^p_f(g).
    \end{align*}
    One proves inductively that 
    \begin{equation*}
        \ad^p_{a\otimes f}(b\otimes g)=\ad^{p-k}_{a\otimes f}\bigl(a^kb\otimes \ad^k_f(g)\bigl),~\forall~ 0\leq k\leq p.
    \end{equation*} Therefore, $\ad^p_{a\otimes f}(b\otimes g)= \bigl[(a\otimes f)^{[p]},b\otimes g\bigl].$
    \end{enumerate}
\subsubsection{Restricted morphisms and modules.}
Let $\left( \fg,[-,-]_\fg,-)^{[p]_\fg}\right) $ and $\left( \fa,[-,-]_\fa,(-)^{[p]_\fa}\right) $ be two restricted Lie algebras. A Lie algebra morphism $\varphi:\fg\longrightarrow \fa$ is called \textit{restricted morphism} (or \textit{$p$-morphism}) if it satisfies 
\begin{equation*}
\varphi\bigl(x^{[p]_\fg} \bigl)=\varphi(x)^{[p]_\fa}\quad\forall x\in \fg.
\end{equation*}

A \textit{restricted representation} of a restricted Lie algebra $\left( \fg,[-,-]_\fg,(-)^{[p]_\fg}\right) $ (also called \textit{restricted $\fg$-module}) is a $k$-vector space $V$ equipped with a restricted morphism $\varphi:\fg\rightarrow\End(V)$, where the restricted Lie structure on $\End(V)$ is given by Example \ref{examplesun} (ii). We denote it by $(V,\varphi)$.

A linear map $d:\fg\rightarrow \fg$ is called a \textit{restricted derivation} of $\fg$  if in addition to the Leibniz condition
 \begin{equation*}
    d\bigl([x,y]_\fg\bigl)=[d(x),y]_\fg+[x,d(y)]_\fg,~\forall x,y\in \fg,
\end{equation*} it also satisfies
\begin{equation*}
    d\bigl(x^{[p]_\fg}\bigl)=\ad_x^{p-1}(d(x)),~\forall x\in \fg.
\end{equation*}
\begin{example} With the notations of Example \ref{examplesun} ($iv$), the linear map
    \begin{equation}\label{eq:JDJderivation}
d_{a,f}:\A\otimes\D\rightarrow\A\otimes\D,\quad d_{a,f}(b\otimes g):=af(b)\otimes g,\,\,a,b \in \mathcal{A},\,f,g\in \mathcal{D}
    \end{equation} is a restricted derivation \emph{when} $\mathcal{A}$ \emph{is commutative}.
The proof is straightforward. Let $a,b \in \mathcal{A}$,
        \begin{align*}
            d_{a,f}\bigl((b\otimes g)^{[p]}\bigl)=d_{a,f}(b^p\otimes g^{[p]})=af(b^p)\otimes g^{p}= a {\rm ad}_b^{p-1}f(b)\otimes g^p.
        \end{align*} 
        (see \cite{J} page 21)
        We also have
        \begin{align*}
            \ad^{p-1}_{b\otimes g}\circ d_{a,f}(b\otimes g)=\ad^{p-1}_{b\otimes g}\bigl(af(b)\otimes g\bigl)=b^{p-1}af(b)\otimes {\rm ad}_g^{p-1}(g).
        \end{align*}
    Then, it has been shown in \cite[Theorem 3.1]{JZ} that the map $d_{a,f}$ is an ordinary derivation, thus the conclusion.
\end{example}
    
\begin{remark} In \cite[Theorem 3.1]{JZ}, it has been shown that in characteristic 0, there exists a post-Lie product on the Lie algebra $\A\otimes\D$ with the bracket \eqref{eq:tensorbracket}. Namely, the post-Lie product is given by
\begin{equation}
    (a\otimes f)\btr(b\otimes g):=d_{a,f}(b\otimes g)=af(b)\otimes g,~\forall a,b\in \A,~\forall f,g\in \D.
\end{equation}
In the sequel, we show that it also carries a trivially restricted post-Lie structure for $p = 2,3$, see Propositions \ref{prop:JZp=2} and \ref{prop:JZp=3}.
\end{remark} 
Let $\fg$ be a restricted Lie algebra. A $\fg$-module $(M,\btr)$ is called \textit{restricted} if
\begin{equation}
(p \text{ terms})\,x\btr(x\btr(\cdots(x\btr m)\cdots) =x^{[p]}\btr m, \quad  \forall x\in \fg,~\forall  m\in M.
\end{equation}
Let $\left( \fg,[-,-]_\fg,(-)^{[p]_\fg}\right) $ and $\left( \fa,[-,-]_\fa,(-)^{[p]_\fa}\right) $ be two restricted Lie algebras. 
\emph{A restricted left Lie module} is a restricted Lie algebra in the category of restricted $\fa$-modules, that is, a triple ($\fg$, $\btr$, $\fa$) where $\btr \colon \fa \otimes_{k} \fg \to \fg$ is a linear map such that, for all $x,y\in \fg$ and for all $a,b\in\fa$,
\begin{align*}
    &a\btr[x,y]=[a\btr x,y]+[x,a\btr y];  && [a,b]\btr x=a\btr(b\btr x)-b\btr(a\btr x);\\
    &a\btr x^{[p]_\fg}=\ad_x^{p-1}(a \btr x);  &&a^{[p]_\fa}\btr x=\underset{p \text{ terms }}{\underbrace{a\btr(a\btr(\cdots(a}}\btr x)).
\end{align*} 
We denote such a restricted left Lie module by a triple $(\fg,\btr,\fa)$.

Let $\left( \fg,[-,-]_\fg,(-)^{[p]_\fg}\right) $ be a restricted Lie algebra. We denote by $\fu(\fg)$ the universal enveloping algebra of the (ordinary) Lie algebra $\fg$ and by $I$ the ideal of $\fu(\fg)$ generated by elements $x^p-x^{[p]_\fg},~x\in\fg.$ The \emph{restricted} enveloping algebra $\fu_p(\fg)$ of $\fg$ is defined by $\fu_p(\fg)=\fu(\fg)/I$, see \cite{J}.

\subsection{Conjecture.} We shall, at this point, early in our exposition, already formulate a conjecture for better reference. As the author have observed in \cite{Do} (see references therein),  a (un-graded) brace algebra $(B, -\{-\},\star)$ over a field $k$ with finite characteristic is a restricted pre-Lie algebra (or, in our terminology, is a trivially restricted pre-Lie algebra). The $p$-structure involved is the $p$ iteration of the pre-Lie product, that is, 
$$
x^{[p]} = x\{x\}\{x\}\cdots \{x\}\,({ p\,\rm times}),\quad \forall x\in B.
$$

We make the following conjecture.
 Assume the following: 
 \begin{enumerate}
 \item $B$ possesses an element $e \in B$ such that $e\{p\} = p$; 
 \item The Lie algebra $(B,[-,-]_{\star})$ is restricted (and we denote by $(-)^{[p]}$ the $p$-structure).
 \end{enumerate}
 Then, the product $x \btr y : = x\{e\star y\}$, $\forall x,y \in B,$ yields a post-Lie operation on the Lie algebra $(B,[-,-]_{\star})$.
 Furthermore, if we define  
\begin{align}
x^{[p]_{\btr}} &:= x\{e\star x\}\cdots \{e \star x\}\\ &+\sum_{k=2}^{p-1}\frac{(-1)^{k-1}}{k}\sum_{\substack{\ell_1 + \cdots + \ell_{k} = p \\ \ell_i \geq 1}} \frac{(p-1)!}{(\ell_1-1)!\cdots (\ell_k-1)!}\frac{1}{C_{\ell_1,\ldots,\ell_k}}\,[x^{\bullet\ell_1} \cdots x^{\bullet \ell_{k}}],
\end{align}
where $x^{\bullet\, \ell} = x\{e\star x^{\bullet\,\ell-1}\}$,
then the adjacent Lie algebra $\bigl(B, [x,y]=x\{e\star y\} - y\{e\star x\}$, $x^{[p]_{\btr}}\bigl)$  is restricted.

%%%%%%%%%%%%%%%%%%%%%%%%%%%%%%%%%%%%%%%%%%%%%%%%%%%%%%%%%%%%%%%%%%%%%%%%%%%%%%%%%%%%%%%%%%%%%%

\section{Jacobson identities for Post-Lie algebras}
 For a restricted Lie algebra $\fa$, we denote by $\pi_{\fa}\colon \mathfrak{u}(\fa)\to\mathfrak{u}_p(\fa)$ the canonical projection onto the restricted enveloping algebra.
 
\subsection{Post-Lie algebras.}\label{sec:postLie} We introduce the basic definitions and terminology about post-Lie algebras and the extension taken from \cite{EFLM} of the Guin-Oudom construction of the envelope of the sub-adjacent Lie bracket to a post-Lie product, see \cite{OG}.
\begin{definition}(see \cite{Va}, \cite{MKL})\label{def:postLie}
A triple $(\fg, [-,-],\btr)$ is called a Post-Lie algebra if 
\begin{enumerate}
\item $(\fg, [-,-])$ is a Lie algebra,
\item for any $x,y,z \in \fg$, we have
$$
x \btr [y,z] = [x \btr y, z] + [y, x \btr z]; \quad \text{and} \quad  [x,y]\btr z = a_{\btr}(x,y,z) -  a_{\btr}(y,x,z),
$$
\end{enumerate}
where $a_{\btr}$ is the associator of $\btr$ defined by $a_{\btr}(x,y,z) = x\btr (y \btr z) - (x\btr y) \btr z,~\forall x,y,z\in \fg$.
\end{definition}
To any post-Lie algebra $(\fg, [-,-],\btr)$ corresponds another binary operation denoted by $\llbracket-,-\rrbracket$ on $\fg$ and defined by 
\begin{equation}\label{eq:subadjacent-bracket}
    \llbracket x,y\rrbracket:=[x,y]+x\btr y-y\btr x,\quad \forall x,y\in\fg.
\end{equation}
This operation is a Lie bracket on $\fg$ called \textit{sub-adjacent Lie bracket}. The Lie algebra $(\fg,\llbracket-,-\rrbracket)$ is called the \textit{sub-adjacent Lie algebra}. Moreover, the map $x\mapsto x\btr(-)$ yields a representation of the sub-adjacent Lie algebra.\\

Let $(\fg, [-,-],\btr)$ be a post-Lie algebra. On $\mathfrak{u}(\fg)$, define a bilinear operation which we still denote with the symbol $\btr$ by the following formulas:
\begin{align}
\label{eqn:extension}
xE \btr = x \btr(E\btr F) - (x \btr E)\btr F,\quad E \btr F_1 F_2 = (E_{(1)} \btr F_{1})(E_{(2)} \btr F_{2}),
\end{align}
with $x \in \fg, E,F,F_1,F_2 \in \mathfrak{u}(\fg)$. Let $\star_{\btr}$ be the bilinear operation on $\mathfrak{u}(\fg)$ defined by 
\begin{equation}
E \star_{\btr} F = E_{(1)} (E_{(2)}\btr F),\quad \forall E,F \in \mathfrak{u}(\fg).
\end{equation}
The product $\star_{\btr}$ is associative and moreover, the map 
\begin{equation}
\mathfrak{u}(\fg, \llbracket -,-\rrbracket) \to \bigl(\mathfrak{u}(\fg), \star_{\btr}, 1, \Delta_{\mathfrak{u}(\fg)}, \varepsilon_{\mathfrak{u}(\fg)}\bigl),\quad x_2\cdots x_n \mapsto x_2 \star_{\btr} \cdots \star_{\btr} x_n
\end{equation}
is an isomorphism of Hopf algebras, see \cite{OG,EFLM}.
\subsubsection{(Ordinary) $\mo$-operators.}
Let $(\fg, \btr, \fa)$ be a  Lie module. Following \cite{K}, an (ordinary) $\mo$-operator on  $(\fg, \btr, \fa)$ is a linear map $\tht \colon \fg \rightarrow \fa$ satisfying
\begin{equation}
\label{eqn:ooperator}
[\theta(x), \theta(y)] = \theta\bigl([x,y] + \tht(x) \btr y - \tht(y) \btr x\bigl), \quad \forall x,y \in \fg.
\end{equation}
Let $\theta$ be an (ordinary) $\mathcal{O}$-operator. We define a Lie bracket called \textit{adjacent bracket} by
\begin{equation}
\llbracket x,y \rrbracket_{\theta} : = [x,y]_{\mathfrak{g}} + \theta(x) \btr y - \theta(y) \btr x, \quad \forall x,y\in \fg.
\end{equation}
Equation \eqref{eqn:ooperator} implies that $\llbracket -,- \rrbracket_{\theta}$ is the sub-adjacent Lie bracket to the post-Lie product defined by
\begin{equation}\label{eq:postlie-from-oop}
x \btr_{\theta} y = \theta(x) \btr y, \quad \forall x,y \in \mathfrak{g}.
\end{equation}
 Let $(\fg,[-,-])$ be a Lie algebra. A \emph{Rota-Baxter operator} on $\fg$ is a linear operator $R\colon \fg \to \fg$ satisfying
\begin{equation*}
[R(x),R(y)] = R\bigl([R(x),y] + [x,R(y)] + [x,y]\bigl), \quad \forall x,y \in \fg.
\end{equation*}
A Rota-Baxter operator is an $\mathcal{O}$-operator on the Lie module $(\fg,{\rm -ad},\fg)$.
Let $\theta\colon \mathfrak{g} \to \mathfrak{a}$ be an $\mathcal{O}$-operator over a Lie module $(\fg,\btr,\fa)$. The Guin-Oudom recursion (see \cite{OG,G}) yields a linear operator 
$\bm{\theta} \colon \mathfrak{u}(\fg) \to \fu(\mathfrak{a})
$ satisfying for any $x \in \mathfrak{g}$ and any $ E \in \mathfrak{u}(\fg):$
\begin{align*}
\bm{\theta}(x) &= \theta(x); \\
\bm{\theta}(xE) &= \theta(x)\bm{\theta}(E) - \bm{\theta}(\theta(x)\btr E).
\end{align*}
With these definitions, the bilinear operation $\btr_{\theta}$ defined in Equation \eqref{eqn:extension} extending the post-Lie product $\btr_{\theta}$ is expressed as
$ \btr_{\theta}\, = {\theta}(-)\btr (-)$ (where in the right-hand side of the previous equation $\btr$ is the integration of the action $\btr$ of $\fg$ on $\fa$ to an action of the bi-algebra $\mathfrak{u}(\fg)$ on $\mathfrak{u}(\fa)$, see \cite{G}).
\begin{example}
\label{sec:braces}
Let $(\mathcal{P},\gamma,I,m)$ be a non-symmetric connected operad with multiplication $m \in \mathcal{P}(2)$. Define 
$$
\mathbb{C}[\mathcal{P}]_0 : = \bigoplus_{n\geq 1} \mathcal{P}(n),\quad \mathbb{C}[\mathcal{P}] : = \mathbb{C}\eta \oplus \bigoplus_{n\geq 1} \mathcal{P}(n),
$$ where $\eta$ is the unit for the associative product on $\mathbb{C}[\mathcal{P}]$ inducted by $m$.
Denote by $\vartriangleright$ the Gerstenhaber-Voronov pre-Lie product (see \cite[Example 3.5.1]{G}) and define a map 
$$
\rho \colon \mathbb{C}[\mathcal{P}] \rightarrow \mathbb{C}[\mathcal{P}]_0,\quad p \mapsto \gamma(m;I,p), \, 1\mapsto I.
$$
Then in \cite[Proposition 3.5.1]{G}, the second author proved that $\rho$ is an $\mathcal{O}$-operator when $\mathbb{C}[\mathcal{P}]_0$ is equipped with the sub-adjacent Lie bracket to the Gertenhaber-Voronov pre-Lie product and $\mathbb{C}[\mathcal{P}]$ is equipped with the commutator Lie bracket of $m$.
\end{example}
\subsection{Jacobson identities for post-Lie algebras.}
\label{sec:jacobson}
 We are ready to state our main Theorem.
\begin{theorem} 
\label{thm:main} 
Let $(\fg, \btr, [-,-])$ be a post-Lie algebra over a field $k$ of characteristic $p>0$. With the notations introduced so far, for any $x \in \fg$, we have
\begin{equation}
\label{eqn:mainstatement}
L(x) := x^{\star_{\btr}p} - x^p  - x^{\bullet_{\btr} p} \in [\fg,\fg],
\end{equation}
where $x^{\bullet_{\btr} p} = x\btr x^{\bullet p-1}$, $x\btr x^{\bullet 0}=x$. Moreover, the following formula holds 
\begin{align}
L(x) = \sum_{\substack{\ell_1 + \cdots + \ell_{n} = p, \\ \ell_i \geq 1,\, n\geq 2,\,\exists \ell_j \geq 2.}} \frac{1}{n}\frac{(p-1)!}{(\ell_1-1)!\cdots (\ell_n-1)!}\frac{1}{C_{\ell_1,\ldots,\ell_n}}[x^{\bullet \,\ell_1} \cdots x^{\bullet \,\ell_{n}}].
\end{align}
\end{theorem}

\begin{proof}
 See Section \ref{sec:proofs}.
\end{proof}

Denote by $Z_{\bullet}(\fg)$ the lower central series of $\fg$. For Rota-Baxter operators, the above Theorem can be made more precise.
\begin{corollary} 
\label{thm:mainrb}
Let $\theta\colon \fg \to \fa$ be a Rota--Baxter operator on $\fg$. With the notation introduced so far, for any $x \in \fg$, we have
\begin{equation}
L(x) \in Z_{p}(\fg).
\end{equation}
where we denote by $x^{\bullet_{\theta} p} = \theta(x)\btr x^{\bullet_{\theta} p-1} $.
\end{corollary}

\subsubsection{Pre-Lie (left-symmetric) Jacobson identities,  see \cite{asqar}.}\label{sec:preliejacobson}
From Theorem \ref{thm:main}, in the case where $[\fg,\fg]=0$ ($\fg$ is the abelian Lie algebra), we have
\begin{equation}
x^{\star_{p}} = x^{p} + x^{\bullet_{\btr} p},\quad\forall x\in\fg.
\end{equation}
Note that in that case, for any $x,y \in \mathfrak{g}$, we have
\begin{equation}
(x+y)^{\bullet_{\btr} p} = (x+y)^{\star_{p}} - (x+y)^{p} = (x+y)^{\star_{p}} - x^{p} - y^{p},
\end{equation}
and the Jacobson identities for $\bullet_{\btr}$ proved in \cite{asqar} (First statement in Theorem 1.1) follow from the associative Jacobson formula for the product $\star_{\btr}$ and for the product on $\mathfrak{u}(\fg)$. We can as well retrieve the last statement of Theorem 1.1 of \cite{asqar}: the map
$$
y \mapsto (x^{\bullet_{\btr} p} - x^{\star_{\btr}p}) \btr y
$$
is a ${\btr}$-derivation. In fact, for any $y,z\in \fg$, we have
\begin{align}
(x^p\btr y) \btr z + y \btr( x^p \btr z) &= x^p\btr y + y\star_{\btr}x^p)\btr z \nonumber \\
&=(x^p \btr y + y\btr x^p + yx^p)\btr z. \label{ref:lastoneA}
\end{align}
Since 
$y\btr x^p = p (y\btr x) x^{p-1} = 0$ and $x^py=yx^p$, we get that $\eqref{ref:lastoneA}$ is equal to 
$$
\eqref{ref:lastoneA}= (x^p \btr y + x^py) \btr z = (x^p \star_{\btr} y) \btr z,
$$
where we have used that $x^p$ is a primitive element in $\mathfrak{u}(\fg)$ (meaning that $\Delta(x^{p})=1\otimes x^{p} + x^{p}\otimes 1$). 

\subsubsection{Post-Lie Jacobson identities.}
\label{sec:postliejacobson}
We generalize the results of the previous section to post-Lie algebras. Let $(\fg,[-,-],\btr)$ be a post-Lie algebra and let $(-)^{[p]}$ be a $p$-map on $(\fg,[-,-])$. Let $(-)^{[p]_{\btr}}$ be the map defined in Equation $\eqref{eqn:pstruct}$ and define for all $x\in \fg$ a map
$$
D(x)\colon \fg\to\fg,\quad D(x)(y): = (x^{\star_{\btr}p}-x^{[p]_{\btr}})\btr y.
$$
\begin{theorem}
\label{prop:jacobsonidentitties}
Assume that the map $x\mapsto y \btr x$ is a restricted derivation of $(\fg, [-,-], (-)^{[p]})$ for any $y\in \fg$. Then, 
\begin{enumerate}
\item for any $x,y\in \fg$, $D(x)$ is a $\btr$-derivation and $[D(x), D(y)] = 0$; 
\item the map $x\mapsto x^{[p]_{\btr}}$ defined by Equation \eqref{eqn:pstruct} satisfies the following Jacobson identities :
\begin{equation}
(x+y)^{[p]_{\btr}} = x^{[p]_{\btr}} + y^{[p]_{\btr}} + \sum_{\substack{(x_k)_k \in \{x,y\}^{p}\\ x_{p-1}=y, x_p =x}} \frac{1}{\sharp(x)} \llbracket x_2,\llbracket x_2, \llbracket \ldots,\llbracket x_{p-1},x_p \rrbracket \cdots \rrbracket\rrbracket \rrbracket,
\end{equation}
where $\sharp(x):=\text{card}\{k,~x_k=x,~k=1,\cdots,p\}$.
\end{enumerate}
\end{theorem}
\begin{proof}
Let $x,y,z \in \fg$. Set $d(x) = x^{\star_{\btr}p}-x^{[p]_{\btr}}= x^{p}-x^{[p]}$. We notice that 
\begin{align}
(d(x) \btr y)\btr z +y \btr  (d(x) \btr z) &= (d(x) \btr y)\btr z + (y \star_{\btr} d(x))\btr z \nonumber \\
&=(d(x) \btr y)\btr z + (y \btr d(x))\btr z + (yd(x))\btr z \nonumber \\
&=(d(x) \btr y)\btr z + (y \btr d(x))\btr z + (d(x)y)\btr z \label{ref:equationdeux}
\end{align}
where, in the last equality, we have used that $ d(x) = x^p-x^{[p]}$ is a central element in $\mathfrak{u}(\fg)$.
Now, we have
\begin{align*}
y \btr d(x) &=  \sum_k x^{p-1-k}(y \btr x) x^k - y \btr x^{[p]} = \ad_x^{p-1}(y \btr x) - y \btr x^{[p]} =0,
\end{align*}
we have used the fact that $\btr$ induces a restricted derivation to conclude. Hence \eqref{ref:equationdeux} is equal to 
$
( d(x) \star_{\btr} y) \btr z.
$
and 1. follows. Next, $x^{[p]_{\btr}} = x^{\star_{\btr}p} + x^{[p]}-x^{p}$. Since $x\mapsto x^{[p]}-x^{p}$ is $p$-linear, 2. follows from the Jacobson identities for the associative product $\star_{\btr}$. 
\end{proof}
\subsection{Proofs.}
\label{sec:proofs} In this section, we prove Theorem \ref{thm:main}. We begin with two preliminaries remarks before giving the proof of the Theorem, which is divided in two parts.
\subsubsection{}
Let $n\geq 1$ be an integer. We let $\mathcal{O}\mathcal{T}(n)$ be the set of rooted trees with $n$ nodes. We denote by $t_1 \btr t_2 $ the left grafting operation of the tree $t_1$ to the root of the tree $t_2$. We let $\mathcal{L}\mathcal{O}\mathcal{T}$ be the free Lie algebra over the field $k$ generated by $\bigcup_{n\geq 1}\mathcal{O}\mathcal{T}(n) \cup \{ \emptyset \}$. We extend $\btr$ as a post-Lie operation to $\mathcal{L}\mathcal{O}\mathcal{T}$. Then $(\mathcal{L}\mathcal{O}\mathcal{T}, [-,-], \btr)$ is the free postLie algebra over one generator, see e.g. \cite{MKL}. The envelope of $\mathcal{L}\mathcal{O}\mathcal{T}$ is the tensor algebra over $k\mathcal{O}\mathcal{T}$, which we denote $k\mathcal{O}\mathcal{F}$.

\subsubsection{}
Let $p\geq 2$ be a prime number. We let $\mathcal{L}\mathcal{O}\mathcal{T}_p(n)$ be the free $p$-Lie algebra generated by $\mathcal{O}\mathcal{T}$ : they are all the elements of $\mathcal{L}\mathcal{O}\mathcal{T}$ and their $p^k$, $k\geq 1$ powers. Then $k\mathcal{O}\mathcal{F}$ is the free restricted enveloping Lie algebra over $k\mathcal{O}\mathcal{T}$.

\subsubsection{Proof of Theorem \ref{thm:main}: first part.}

Let $\bar{\delta} \colon k\mathcal{O}\mathcal{F} \rightarrow k\mathcal{O}\mathcal{F} \otimes_{k} k\mathcal{O}\mathcal{F}$ be the co-augmentation part of the unshuffle co-product (the reduced co-product). Let $x\in \mathcal{L}\mathcal{O}\mathcal{T}$. Then already
$
\bar{\delta}(x^{\star_{\btr}p}) = 0
$
yields that $x^{\star_{\btr} p}$ is an element of the free restricted Lie algebra $k\mathcal{L}\mathcal{O}\mathcal{T}_{p}$ (see \cite{Re}). But then, we have
$$
x^{\star_{\btr p}} = x^{p} + u,
$$
where $u$ is a word on elements in $\mathcal{L}\mathcal{O}\mathcal{T}$ of length less than $p-1$. Since $\bar{\delta}(u)=0$, one gets that $u$ is in the free Lie algebra $k\mathcal{L}\mathcal{O}\mathcal{T}$.

\subsubsection{Proof of Theorem \ref{thm:main}: second part.}

We let $\mathcal{O}\mathcal{F}(n)$ be the set of all ordered forests with $n$ nodes. We denote by $V(f)$ the set of nodes of a forest $f$.
A linearisation of an ordered forest $f$ is a function $\ell \colon V(f) \to I^{\prec}$ \emph{decreasing for the hereditary order and for the right-to-left linear order} between the nodes attached to the same corolla in $f$ (including the roots of the trees in the forest).
We denote by $\mathcal{O}\mathcal{F}_{\ell}$ the set of linearizations of ordered forests and by $\mathcal{O}\mathcal{C}_{\ell}$ the sub-set of $\mathcal{O}\mathcal{F}_{\ell}$ made of the linearizations of forests of corollas.
\begin{proposition} With the notations introduced so far, let $\tau_1,\ldots,\tau_{n-1}, \tau_{n} \in \mathcal{O}\mathcal{T}$ be rooted planar trees. Then,
\label{prop:formulathetaproduct}
\begin{align}
\label{eqn:seconstatement}
\tau_1 \star_{\btr} \cdots \star_{\btr }\tau_n = \sum_{(c,\ell) \in \mathcal{O}\mathcal{C}_{\ell}(n)} {(c,\ell)}\circ (\tau_1, \ldots, \tau_n),
\end{align}
where $(c,\ell) \circ (\tau_1, \ldots, \tau_n)$ is obtained by substituting to the vertex of $c$ labeled $i$ in $\ell$ by the tree $\tau_i$ and grafting to the left the trees adjacent to that vertex in $f$ to $\tau_i$ (while keeping the ordering of those trees).
\end{proposition}
\begin{proof} We prove the statement of the proposition by induction on $n$, when $n=0$ there is nothing to prove and with $n=1$ the statement follows from the definition of the product $\star_{\btr}$. Let us suppose the statement is true at the rank $n-1$.
\begin{align}
\tau_1 \star_{\btr} \tau_2\star_{\btr} \cdots \tau_{n} &= 
\tau_1 \star_{\btr} (\tau_2 \star_{\btr} \cdots \star_{\btr}\tau_{n}) \\
& = \tau_1 \star_{\btr} \Big( \sum_{(c,\ell) \in \mathcal{O}\mathcal{C}_{\ell}(n)} {(c,\ell)}\circ (\tau_1, \ldots, \tau_n) \Big) \\
& = \sum_{(c,\ell) \in \mathcal{O}\mathcal{C}_{\ell}(n)} \tau_1 \,\,{(c,\ell)}\circ (\tau_2, \ldots, \tau_n) + \tau \btr ({(c,\ell)}\circ (\tau_2, \ldots, \tau_n)) \\
& = \sum_{(c,\ell) \in \mathcal{O}\mathcal{C}_{\ell}(n)} {(c',\ell)}\circ (\tau_1,\tau_2, \ldots, \tau_n) + \sum_{(c,\ell) \in \mathcal{O}\mathcal{C}_{\ell}(n)} \tau \btr ({(c,\ell)}\circ (\tau_2, \ldots, \tau_n)),
\end{align}
where $(c',\ell)$ is obtained from $(c,\ell)$ by adding to the left of the forest $c$ a root labeled $1$. Write $c = c_2\cdots c_n$. By using the fact that $\btr$ is a derivation, we may expand the general term of the second sum as 
\begin{align}
&\sum_{i=1,\cdots, k} (c_2, \ell_{|V({c_2})})\circ (\tau_{\ell_{|V(c_2)}})\cdots (c_{i-1}, \ell_{|{V(c_{i-1})}})\circ (\tau_{\ell_{|V(c_{i-1})}}) \\ 
&\hspace{1.5cm}(\bullet \btr c_i,\,\ell_{|V(c_i)}) \circ (\tau_1,\tau_{\ell_{|V(c_i)}}) (c_{i+1}, \ell_{|V({c_{i+1}})}) \circ (\tau_{\ell_{|V(c_{i+1})}})\cdots  (c_{k}, \ell_{|V({c_{k})}}) \circ (\tau_{\ell_{|V(c_k)}}),
\end{align}
where, again, $\ell_{|c_i}$ is the restriction of the labeling $\ell$ to the corolla $c_i$ in $c$ and $\ell_{|c_i}$ labels the leftmost vertex in the corolla $\bullet \btr c_i$ by $1$.
 We end the proof by noticing that the map
 $$
(c',\ell) \mapsto ((c,\ell), i(c',\ell)),
 $$
is a bijection from $\mathcal{O}\mathcal{F}_{\ell}$ to its image, with $c$ is the corolla obtained from $c'$ by removing the vertex labeled $1$, $\ell$ the restriction of $\ell$ to $V(c)$ and $i(c',\ell)$ is the index of the corolla containing the node labeled $1$.
\end{proof}
Equation \eqref{eqn:seconstatement} yields, when $\tau_1 =\cdots = \tau_n = \bullet \in \fg $,
\begin{align*}
\bullet^{\star_{\btr} n} = \sum_{c \in  \mathcal{C}(n)} \frac{n!}{\prod_{v \in V(c)} | v \to_c \cdot| } c,
\end{align*}
where $v \to \cdot$ is equal to all nodes in $f=c_1 \cdots c_n$ less than $v$ in the hereditary plus right-to-left order. 
The coefficient ${\prod_{v \in V(c)} | v \to_c \cdot| }$ is given by
$$
{\prod_{v \in V(c)} | v \to_c \cdot| } = (\ell_k-1)! \cdots (\ell_1-1)!(\ell_1+\cdots+\ell_k)(\ell_1+\cdots+\ell_{k-1})\cdots \ell_1.
$$
\subsection{Cancellations in formula \eqref{eqn:pstruct}}
\label{sec:simplifications} In this section, we compute the coefficients that appear in the main formula \eqref{eqn:pstruct} and expands  $x^{[p]_{\btr}}$ over a certain family of right-iterated brackets.

\subsubsection{Equivalence classes.}
Let $n\geq 1$ be a positive integer and denote by $\mathcal{S}_n$ the symmetric group. Given two tuples $(\ell_1,\ldots,\ell_n)$ and $(\ell_1',\ldots,\ell_n')$ as in formula \eqref{eqn:pstruct}, we say that they are \emph{equivalent} if there exists $\sigma\in\mathcal{S}_n$ such that $\ell'_{\sigma(i)}=\ell_i,$ for any $1\leq i\leq n$.

% for any $1 \leq \ell \leq p$, we have
% $$
% |\{1\leq j\leq n: \ell_j = \ell \}| = |\{1\leq j\leq n: \ell_j' = \ell \}|.
% $$ 
Equivalence classes of tuples $(\ell_1,\ldots,\ell_n)$ with $\ell_1+\cdots+ \ell_n = p$, $\ell_i \geq 1$ are parametrized by integer partitions $\lambda = 1^{m_1}\cdots p^{m_p}$, $m_i \geq 0$  of $p$ (we will write $\lambda\vdash p$). We also denote by $\lambda$ the tuple $(k:m_k\neq 0)$.

Given a tuple $(\ell_1 \leq \cdots \leq \ell_n)$ as in \eqref{eqn:pstruct}, we denote by $L(\{x_{\ell_1},\ldots, x_{\ell_n}\})$ the free Lie algebra generated by $\{x_{\ell_1},\ldots, x_{\ell_n}\}$ (for example, if $\ell_1=\ell_2=1, \ell_3=2$, we consider the free Lie algebra $L(\{x_1,x_2\})$ generated by all the bracketings of $x_{1},x_{2}$). 

Let us fix an integer partition $\lambda \vdash p$ and set $n = \sum_k m_k$ for the length of the partition integer $\lambda$. We set further
\begin{align}
\label{eqn:pl}
P_{\lambda}(x)&= \frac{1}{n}\frac{(p-1)!}{0!^{m_1}\cdots (p-1)!^{m_p}} \sum_{\substack{\lambda \,\sim\, (\ell_1,\ldots,\ell_n) \\ \ell_1 + \cdots + \ell_{n} = p \\ }} {C^{-1}_{\ell_1,\ldots,\ell_n}}\,[x_{\ell_1} \cdots x_{\ell_{n}}] \in {\mathfrak{g}}.
\end{align}
We use ${\rm Sh}(n-s,s)$ for the set of all shuffle permutations of $[n] = \{1,\ldots,n-s\}\cup \{n-s+1,\ldots,n\}$: 
\begin{align}
{\rm Sh}(n-s,s) = \{ \sigma\colon [n]\to [n] ~\colon~ \sigma(n-s) > \cdots > \sigma(1), ~  \sigma(n) > \cdots > \sigma(n-s+1) \} , 
\end{align}
and ${\rm Stab}(\lambda) = \mathcal{S}_{1}^{m_1}\times \cdots \times \mathcal{S}_{p}^{m_p}$ is the stabilizer of $\lambda$ for the left action of $\mathcal{S}_n$ on tuples with size $n\geq 1$.
Friedrich's criterion (see \cite{Re}) yields the following Lemma. 
\begin{lemma}\label{lem:friedrich}
Let $p,\lambda,n$ be as above and assume $n\geq 2$. For any integer $1\leq s \leq n-1$ and tuples $(\ell_1,\ldots,\ell_n)$ with type $\lambda$, 
\begin{equation}
\label{eqn:friedrichcriterion}
\sum_{\alpha \in {\rm Sh}(s,n-s)} C^{-1}_{\alpha\cdot\ell_1,\ldots,\alpha\cdot\ell_n}=0.
\end{equation}
\end{lemma}
\begin{proof} Let $\langle - | - \rangle$ be the symmetric bilinear product on $\mathbb{C}\{x_{\ell_1},\ldots,x_{\ell_n}\}^{\star}$ defined by the requirement that $\{x_{\ell_1},\ldots,x_{\ell_n}\}$ is an orthonormal basis.
Friedrich's criterion implies that for any non-empty word $v$ with length greater than $2$:
\begin{align*}
0&=\sum_{\alpha \in {\rm Sh}(s,n-s)}\sum_{\substack{\ell_1 + \cdots + \ell_{n} = p \\ \lambda \,\sim\, (\ell_1,\ldots,\ell_n)}} {C^{-1}_{\ell_1,\ldots,\ell_n}}\, \langle \, x_{\ell_1} \cdots x_{\ell_{n}} \,|\, \alpha\cdot v \,\rangle \\ 
&= \sum_{\alpha \in {\rm Sh}(s,n-s)}\sum_{\substack{\ell_1 + \cdots + \ell_{n} = p \\ \lambda \,\sim\, (\ell_1,\ldots,\ell_n)}} {C^{-1}_{\ell}}\, \langle \, x_{\alpha^{-1}\cdot \ell_1} \cdots x_{\alpha^{-1}\cdot\ell_{n}} \,|\, v \,\rangle. 
\end{align*}
By choosing $v = \beta \cdot \lambda$, we obtain: 
\begin{align*}
0&= \sum_{\alpha\in{\rm Sh}(s,n-s)} \sum_{\substack{\ell_1 + \cdots + \ell_{n} = p \\ \lambda \,\sim\, (\ell_1,\ldots,\ell_n)}} {C^{-1}_{\ell}} \, \langle x_{\beta^{-1}\alpha^{-1}\cdot \ell_1} \cdots x_{\beta^{-1}\alpha^{-1}\cdot\ell_{ n}} \,|\, x_{\lambda_1}\cdots x_{\lambda_n} \rangle \\
&= \sum_{\alpha\in{\rm Sh}(s,n-s)} \sum_{\gamma \in {\rm Sh}(\lambda)} {C^{-1}_{\gamma \cdot\lambda}} \, \langle x_{\beta^{-1}\alpha^{-1}\gamma \cdot \lambda_1} \cdots x_{\beta^{-1}\alpha^{-1}\gamma\cdot\lambda_{n}} \,|\, x_{\lambda_1}\cdots x_{\lambda_n} \rangle \\
&= \sum_{\substack{\alpha\in{\rm Sh}(s,n-s) \\ \gamma \in {\rm Sh}(\lambda)\\ \beta^{-1}\alpha^{-1}\gamma \in {\rm Stab}{(\lambda)}}} {C^{-1}_{\gamma \cdot\lambda}} \\
& = \sum_{\substack{\alpha\in{\rm Sh}(s,n-s) \\ \gamma \in {\rm Sh}(\lambda)\\ \gamma \in \alpha\beta{\rm Stab}{(\lambda)}}} {C^{-1}_{\alpha\beta \cdot\lambda}} \\
&= \sum_{\alpha\in{\rm Sh}(s,n-s)} \sum_{\substack{\gamma \in {\rm Sh}(\lambda)\\ \gamma \in \alpha\beta{\rm Stab}{(\lambda)}}} {C^{-1}_{\gamma \cdot\lambda}} = \sum_{\alpha\in{\rm Sh}(s,n-s)} \sum_{\substack{\gamma \in {\rm Sh}(\lambda)\\ \gamma \in \alpha\beta{\rm Stab}{(\lambda)}}} {C^{-1}_{\alpha\beta \cdot\lambda}}\\
&= \sum_{\alpha\in{\rm Sh}(s,n-s)} {C^{-1}_{\alpha\beta \cdot\lambda}} \sum_{\substack{\gamma \in {\rm Sh}(\lambda)\\ \gamma \in \alpha\beta{\rm Stab}{(\lambda)}}} 1\\
&= \sum_{\alpha\in{\rm Sh}(s,n-s)} {C^{-1}_{\alpha\beta \cdot\lambda}},
\end{align*}
where the last equality follows from the fact that for each permutation $\alpha \beta$, there exists an unique couple $(s,p)$ in ${\rm Sh}(\lambda)\times {\rm Stab}(\lambda)$ so that $\alpha\beta =  sp$. Hence, $\alpha\beta {\rm Stab}(\lambda) \cap {\rm Sh}(\lambda)= \{s\} $
\end{proof}
\begin{remark}
By choosing instead $v = u_k\beta  \cdot \lambda $, where $u_k$ is the permutation equal to identity on $\{1,\ldots, s\}$ and sending $s+j$ to $n-j$, we obtain 
$$
0 = \sum_{\alpha\in{\rm Sh}(s,n-s)} {C^{-1}_{\alpha u_k\beta \cdot\lambda}} = \sum_{\alpha\in\overline{\rm Sh}(s,n-s)} {C^{-1}_{\alpha\beta \cdot\lambda}},
$$
where $\overline{\rm Sh}{(n-s,s)}$ is the subset of shuffle permutations of $\{1,\ldots,n-s\}\cup \{n < n-1 < \cdots < n-s+1\}$ (note the reversing in the order of the elements).
\end{remark}
Next, we compute $P_{\lambda}(x)$ in the two extreme cases: 
\begin{enumerate}
\item $\lambda$ has only one part not equal to $1$ (this means : $m_1=p-2$ and $m_2=1$) (we call such partition \emph{a hook}) or
\item each one of the $n$ parts has a different size (this means $m_i=1$, $n \geq i \geq 1$).
\end{enumerate}
\subsubsection{The hook case.} We begin with the hook case. We assume hereforth $p \geq 2$.
Denote by
$$
(1,\ldots,\underset{\rm position\,j}{2},\ldots,1) = \varepsilon^j,\quad 1\leq j \leq p-1.
$$ From the definition of right-ordered bracket $[\,\cdots]$, we infer
$$
[x_{\varepsilon^j_1} \cdots x_{\varepsilon^j_{p-1}}] = 0, \quad j\leq p-3.
$$
Thus, 
\begin{align*}
P(1^{p-2}2) &= \frac{1}{p-1}{(p-1)!}(C_{1^{p-2}2}^{-1} - C_{1^{p-3}21}^{-1})[x_2,\ldots,x_2,x_2]\\
&=-{(p-1)!}\bigl((p-2)!^{-1} - (p-3)!^{-1}(p-1)^{-1})\bigl)[x_1,\ldots,x_1,x_2] \\
&=-{\bigl((p-1) - (p-2)\bigl)}[x_1,\ldots,x_1,x_2]\\
&=-[x_1,\ldots,x_1,x_2].
\end{align*}
For any integer $2\leq n \leq p-1$, we compute the value of \eqref{eqn:pl} if $\lambda = 1^{n-1}(p-n+1)$.
\begin{align*}
P(1^{n-1}(p-n+1))&=\frac{1}{n}\frac{(p-1)!}{(p-n)!}(C^{-1}_{1^{n-1}(p-n+1)} - C^{-1}_{1^{n-2}(p-n+1)1})[x_1\cdots x_1,x_{p-n+1}] \\
&=\frac{1}{n}\frac{(p-1)!}{(p-n)!}\bigl((n-1)!^{-1} - (n-2)!^{-1}(p-1)^{-1}\bigl)[x_1\cdots x_1,x_{p-n+1}]\\
&=\frac{1}{n}\frac{(p-1)!}{(p-n)!}\bigl((n-1)!^{-1} + (n-2)!^{-1}\bigl)[x_1,\ldots, x_1, x_{p-n+1}]\\
&=\frac{1}{n}\binom{p-1}{n-1}\bigl(1+(n-1)\bigl)[x_1,\ldots,x_1,x_{p-n+1}] \\
&=(-1)^{n}[x_1,\ldots,x_1, x_{p-n+1}].
\end{align*}

\subsubsection{The case where all the parts have a different size.}
Recall that a basis of the subspace of the free Lie algebra $L(x_{\ell_1},\ldots,x_{\ell_n})$ formed by all Lie words in which exactly one occurrence of each generator appears is given by 
$$
[x_{\sigma\cdot \ell}]:=[x_{\sigma\cdot\ell_{1}} \cdots x_{\sigma \cdot \ell_{n-1}} x_{\ell_n}],\quad \sigma \in \mathcal{S}_{n}\,{\rm with}\,\, \sigma(n)=n.
$$
\begin{proposition}
\label{prop:trucnul}
Let $n\geq 1$ and $1\leq j\leq n-1$. Let $c\in {\rm Sh}(n-1,1)$ be the shuffle permutation sending $n$ to $j$. With the notations introduced before,
$$
[x_{c\cdot \ell}] = [x_{\ell_1}\cdots x_{\ell_{j-1}} x_{\ell_n} x_{\ell_j} \cdots x_{\ell_{n-1}}] = \sum_{\substack{1\leq s \leq n-1\\\sigma \in \overline{\rm UnSh}{(n-s,s)} \\ \sigma c(n)=n \\ \sigma^{-1}(n-s+1)\neq n}} (-1)^{s}[x_{\sigma c \cdot \ell}],
$$
where $\overline{\rm UnSh}{(n-s,s)}$ is the subset of unshuffle permutations of $\{1,\ldots,n-s\}\cup \{n < n-1 < \cdots < n-s+1\}$ (note the reversing in the order of the elements). 
\end{proposition}
\begin{proof} We prove the result by induction on $n\geq 1$ and $1\leq j \leq n-1$. The result trivially holds for $n=2$. Assume that the result holds for all $n\leq N$. When $j=n-1$, on one hand
\begin{align}
[x_{\ell_1}\cdots x_{\ell_n}x_{\ell_{n-1}}] = -[x_{\ell_1}\cdots x_{\ell_{n-1}}x_{\ell_n}]
\end{align}
and on the other hand, 
\begin{align}
\sum_{\substack{1\leq s \leq n-1\\ \sigma \in \overline{\rm UnSh}{(n-s,s)} \\ \sigma(n-1)=n \\ \sigma^{-1}(n-s+1)\neq n}} (-1)^{s}[x_{\sigma c\cdot \ell}] & = \sum_{\substack{2\leq s \leq n-1\\\sigma \in \overline{\rm UnSh}{(n-s,s)} \\ \sigma(n-1)=n \\ \sigma^{-1}(n-s+1)\neq n}} (-1)^{s}[x_{\sigma c\cdot \ell}] -\sum_{\substack{\sigma \in \overline{\rm UnSh}(n-1,1)\\ \sigma(n-1)=n}}[x_{\sigma c \cdot \ell}] \\ 
&= \sum_{\substack{2\leq s \leq n-1\\\sigma \in \overline{\rm UnSh}{(n-s,s)} \\ \sigma(n-1)=n \\ \sigma^{-1}(n-s+1)\neq n}} (-1)^{s}[x_{\ell_{\sigma c}}] -[x_{\ell_1},\ldots,x_{\ell_{n-1}},x_{\ell_{n}}] \label{eqn:sumtodiscussone}.
\end{align}
the sum in the right-hand side of the last equality is empty. In fact, for any $\sigma \in {\rm \overline{UnSh}}$$(n-s,s)$ with $s\geq 2$ and $\sigma(n-1)=n$, we have $\sigma^{-1}(n-1) > \sigma^{-1}(n) = n-1 $ implies $\sigma^{-1}(n-1)=n$. Hence, in the right-hand side of \eqref{eqn:sumtodiscussone}, in the partial sum over $2\leq s \leq n-1$, only the term corresponding to $s=2$ contributes. When $s=2$, the condition $\sigma^{-1}(n-2+1)\neq 2$ and $\sigma(n-1)=n$ are incompatible.
Assume the result holds for $p+1 \leq j \leq N-1$, $p\geq 1$. The Jacobi identity implies:
\begin{align*}
&[x_{\ell_1}\cdots x_{\ell_{p-1}} x_{\ell_n} x_{\ell_p} \cdots x_{\ell_{n-1}}] = [x_{\ell_1} \cdots x_{\ell_{p-1}} [x_{\ell_n},x_{\ell_p}] x_{\ell_{p+1}} \cdots x_{\ell_{n-1}}] \\ 
&\hspace{7cm}+[x_{\ell_1}\cdots x_{\ell_{p}} x_{\ell_n} x_{\ell_{p+1}} \cdots x_{\ell_{n-1}}]. 
\end{align*}
We use the induction hypothesis to infer (see the definition of $\hat{x}$ and $\hat{c}$ after the computations), 
\begin{align}
[x_{\ell_1}\cdots x_{\ell_{p-1}} [x_{\ell_n}, x_{\ell_p}] \cdots x_{\ell_{n-1}}]
&= - [x_{\ell_1}\cdots x_{\ell_{p-1}} [x_{\ell_p}, x_{\ell_n}] x_{\ell_{p+1}} \cdots x_{\ell_{n-1}}] \\
& = - \sum_{\substack{1\leq s \leq n-2\\\sigma \in \overline{\rm UnSh}{(n-1-s,s)} \\ \sigma \hat{c}(n-1)=n-1 \\ \sigma^{-1}(n-s)\neq n-1}}[\hat{x}_{(\sigma \hat{c}) \cdot \hat{\ell}_1}, \ldots, \hat{x}_{(\sigma \hat{c}) \cdot \hat{\ell}_{n-2}}, \hat{x}_{(\sigma \hat{c}) \cdot \hat{\ell}_{n-1}}] \label{eqn:proof:cancellationsun}
\end{align}
with $(\hat{x}_{\hat{\ell}_i})_{n-1 \geq i\geq 1}$ defined by $\hat{x}_{\hat{\ell}_{i}} = x_{{\ell}_{i}}$, $1\leq i \leq p-1$, $\hat{x}_{\hat{\ell}_{k}} = {x}_{{\ell}_{k+1}}$, $p \leq k \leq n-2$, $\hat{x}_{\hat{\ell}_{n-1}} = [x_{\ell_p},x_{\ell_n}]$. The permutation $\hat{c} \in \mathcal{S}_{n-1}$ is such that $\hat{x}_{\hat{c}\cdot \hat{\ell}_1}\cdots \hat{x}_{\hat{c}\cdot \hat{\ell}_{n-1}} = \hat{x}_{\hat{\ell}_1} \cdots \hat{x}_{\hat{\ell}_{p-1}}\hat{x}_{\hat{\ell}_{n-1}}\hat{x}_{\hat{\ell}_{p}} \cdots \hat{x}_{\hat{\ell}_{n-2}}.
$ We replace in \eqref{eqn:proof:cancellationsun} $\hat{x}_{\hat{\ell}_{n}}$ by its definition to obtain
\begin{align}
\eqref{eqn:proof:cancellationsun} &= - \sum_{\substack{1\leq s \leq n-2\\\hat{\sigma} \in \overline{\rm UnSh}{(n-1-s,s)} \\ \hat{\sigma} \hat{c}(n-1)=n-1 \\ \hat{\sigma}^{-1}(n-s)\neq n-1}}[\hat{x}_{\hat{\sigma} \hat{c} \cdot \hat{\ell}_1}, \ldots, \hat{x}_{\hat{\sigma} \hat{c} \cdot \hat{\ell}_{n-2}}, {x}_{{\ell}_p}, {x}_{{\ell}_n}].
\end{align}
Now, with the change of variable defined by 
$$
\hat{\sigma} \mapsto \sigma, \quad 
\begin{cases} 
\sigma(i) = \hat{\sigma}(i),& 1 \leq i \leq p-1  \\
\sigma(p)=\hat{\sigma}(p)+1 = n\\
\sigma(p+1) = n-1 \\
\sigma(j) = \hat{\sigma}(j-1) & p+2 \leq j \leq n,
\end{cases}
$$
we obtain :
\begin{align}
\eqref{eqn:proof:cancellationsun}&= - \sum_{\substack{1\leq s \leq n-1\\\sigma \in \overline{\rm UnSh}{(n-s,s)} \\ \sigma {c}(n)=n \\ \sigma c(p)=n-1 \\ \sigma^{-1}(n-s+1)\neq n}}[{x}_{(\sigma {c}) \cdot {\ell}_1}, \ldots, {x}_{(\sigma {c}) \cdot {\ell}_{n-2}}, {x}_{(\sigma c) \cdot {\ell}_{n-1}}, {x}_{(\sigma c) \cdot {\ell}_n}].
\end{align}
Moreover, we have
\begin{align}
[x_{\ell_1} \cdots x_{\ell_{p-1}}x_{\ell_p}x_{\ell_n}\cdots x_{\ell_{n-1}}] & = \sum_{\substack{1\leq s \leq n-1\\ \tilde{\sigma} \in \overline{\rm UnSh}{(n-s,s)} \\ \tilde{\sigma} \tilde{c}(n)=n \\ \tilde{\sigma}^{-1}(n-s+1)\neq n}} [x_{(\sigma \tilde{c})\cdot\ell_1},\ldots,x_{(\sigma \tilde{c})\cdot\ell_n}] \\
& = \sum_{\substack{1\leq s \leq n-1\\\sigma \in \overline{\rm UnSh}{(n-s,s)} \\ \tilde{\sigma}\tau{c}(n)=n \\ \tilde{\sigma}^{-1}(n-s+1)\neq n}} [x_{(\sigma\tau {c})\cdot\ell_1},\ldots,x_{(\sigma \tau {c})\cdot\ell_n}].\label{eqn:bidule}
\end{align}
Next, we make the change of variables $\tilde{\sigma} \mapsto \sigma,\,  \sigma = \tilde{\sigma}\tau$. Since $\sigma^{-1}(n-k) = \tau \tilde{\sigma}^{-1}(n-k) = \tilde{\sigma}^{-1}(n-k)$ for $0 < k\leq s$ as $\tilde{\sigma}^{-1}(n-k) >  p+1$ and $\sigma^{-1}(n) = p$, we get $\sigma \in \overline{\rm UnSh}{(n-s,s)}$. Also, $\sigma c (p) = \sigma(p+1) = \tilde{\sigma}(p)\leq n-s$. We infer
\begin{align}
\eqref{eqn:bidule}& = \sum_{\substack{1\leq s \leq n-1\\\sigma \in \overline{\rm UnSh}{(n-s,s)} \\ \sigma{c}(n)=n \\ \sigma c(p)\leq n-s \\ \sigma^{-1}(n-s+1)\neq n}} [x_{(\sigma {c})\cdot\ell_1},\ldots,x_{(\sigma {c})\cdot\ell_n}].
\end{align}
The induction step is complete and the proof is done.
\end{proof}
Let us going back to the computations of $P_{\lambda}(x)$. By using Proposition \ref{prop:trucnul}, we infer
\begin{align*}
P_{\lambda}(x) &= \frac{1}{n} \frac{(p-1)!}{\ell_1!\cdots \ell_n!} \sum_{s=1\ldots n-1}\sum_{\substack{\alpha \in S_n\\ 1\leq \alpha(n)<n\\ \sigma \in \overline{\rm UnSh}{(n-s,s)} \\ \sigma \alpha = \beta \\ \beta(n)=n \\ \sigma^{-1}(n-s+1)\neq n }} {(-1)^{s}}{C^{-1}_{\alpha\cdot\ell}}[ x_{\beta \cdot\ell}]
+\frac{1}{n} \frac{(p-1)!}{\ell_1!\cdots \ell_n!}\sum_{\substack{\alpha \in S_n \\  \alpha(n)=n}} C^{-1}_{\alpha\cdot\ell}[x_{\alpha\cdot\ell}] \\
&=\frac{1}{n}\frac{(p-1)!}{\ell_{1}!\cdots \ell_{n}!}\sum_{s=1\ldots n-1}(-1)^s\sum_{\substack{\sigma \in \overline{\rm Sh}(n-s,s)\\ \beta(n)=n\\  \sigma(n)<n \\ \sigma^{-1}(n-s+1)\neq n}} {C^{-1}_{\sigma\beta \cdot \ell}}[ x_{\beta \cdot\ell} ] +\frac{1}{n} \frac{(p-1)!}{\ell_1!\cdots \ell_n!}\sum_{\substack{\alpha \in S_n \\  \alpha(n)=n}} C^{-1}_{\alpha\cdot\ell}[x_{\alpha\cdot\ell}]
\end{align*}
If $ n-1 \geq s \geq 1$, by applying Lemma \ref{lem:friedrich}, we obtain 
\begin{align*}
\sum_{\substack{\sigma \in \overline{\rm Sh}(n-s,s)\\ \sigma(n)<n \\ \sigma(n-s+1)\neq n}} {C^{-1}_{\sigma\beta \cdot \ell}} = \sum_{\substack{\sigma \in \overline{\rm Sh}(n-s,s)\\ \sigma(n-s+1)\neq n}} {C^{-1}_{\sigma\beta \cdot \ell}} = - \sum_{\substack{\sigma \in \overline{\rm Sh}(n-s,s)\\ \sigma(n-s+1)=n}} {C^{-1}_{\sigma\beta \cdot \ell}}
\end{align*}
and we infer the following formula for $P_{\lambda}(x)$:
\begin{align*}
P_{\lambda}(x) &= \frac{1}{n} \sum_{\substack{\beta \in \mathcal{S}_{n}\\ \beta(n)=n}} \frac{(p-1)!}{\ell_1!\cdots \ell_n!}\sum_{s=1\ldots n-1} (-1)^{s+1}\sum_{\substack{\sigma \in \overline{\rm Sh}(n-s,s)\\ \sigma(n-s+1)=n}} C^{-1}_{\sigma\beta\cdot\ell} [x_{\beta\cdot\ell}] \\
&\hspace{2cm}+ \frac{1}{n} \sum_{\substack{\beta \in \mathcal{S}_{n}\\ \beta(n)=n}} \frac{(p-1)!}{\ell_1!\cdots \ell_n!}C^{-1}_{\beta\cdot\ell}[x_{\beta\cdot\ell}] \\
&= \frac{1}{n}\frac{(p-1)!}{\ell_1!\cdots \ell_n!} \sum_{\substack{\beta \in \mathcal{S}_{n}\\ \beta(n)=n}} C^{-1}_{\beta\cdot\ell}+\sum_{s=1\ldots n-1} (-1)^{s+1}\sum_{\substack{\sigma \in \overline{\rm Sh}(n-s,s)\\ \sigma(n-s+1)=n}} C^{-1}_{\sigma\beta\cdot\ell} \quad [x_{\beta\cdot\ell}]
\end{align*}

\subsubsection{The general case.}
Let $\lambda = 1^{m_1}2^{m_2}\cdots p^{m_p}$ be an integer partition of $p$ with $n$ parts. Consider the subset ${\rm Sh}({\lambda})^{+}$ of shuffles $\mathcal{S}_{n}$ relatively to the partition into intervals of $[n]$ equal to $\{ [m_j+1, m_{j+1}],\, 0\leq j \leq k\}$, with the usual conventions, and such that 
\begin{enumerate}
\item $\sigma(n)=n$,
\item $\sigma^{-1}(n-1) \leq m_0+\cdots+m_{k-1}$.
\end{enumerate}
Consider the family $\{ [x_{\sigma \cdot \ell}] \in {\rm Sh}({\lambda})^{+} \}$ of elements in $L(\{x_{\ell_1},\ldots,x_{\ell_n}\})$. Note that $[x_{\sigma \cdot \ell}]=0$ if $\sigma^{-1}(n-1)\geq m_0+\cdots+m_{k-1}$. Our goal is to expand $P_{\lambda}(x)$ in ${\rm Sh}(\lambda)^{+}$.
\begin{align}
P_{\lambda}(x) &= \frac{1}{n}\frac{(p-1)!}{1!^{m_1}\cdots  p!^{m_p}} \sum_{\substack{(\ell_1,\ldots,\ell_n)\sim \lambda \\ \ell_1+\cdots+\ell_n=p}} C^{-1}_{\ell_1,\ldots,\ell_n}[x_{\ell_1,\ldots,\ell_n}]\\ 
&= \frac{1}{n}\frac{(p-1)!}{1!^{m_1}\cdots  p!^{m_p}} \sum_{\alpha \in {\rm Sh}(\lambda)} C^{-1}_{\alpha \cdot \lambda}[x_{\alpha \cdot \lambda}].
\end{align}
We use the splitting $\mathcal{S}_{n} \simeq {\rm Sh}(\lambda)\times {\rm Stab}(\lambda)$. Since $C^{-1}_{\alpha \cdot \lambda}[x_{\alpha \cdot \lambda}] = C^{-1}_{\lambda}[x_{\lambda}]$ for any $\alpha\in {\rm Stab}(\lambda)$, we infer
\begin{align}
P_{\lambda}(x) &=\frac{1}{|{\rm Stab}(\lambda)|n}\frac{(p-1)!}{1!^{m_1}\cdots  p!^{m_p}} \sum_{\alpha \in \mathcal{S}_n} C^{-1}_{\alpha\cdot\lambda} \,\,[x_{\alpha\cdot\lambda}].
\end{align}
By applying Proposition \ref{prop:trucnul}, we obtain
\begin{align}
P_{\lambda}(x)&= \frac{1}{|{\rm Stab}(\lambda)|n}\frac{(p-1)!}{1!^{m_1}\cdots  p!^{m_p}}\sum_{\alpha \in \mathcal{S}_n} C^{-1}_{\alpha \cdot \lambda }\sum_{\substack{n-1 \geq s\geq 1\\  
 \sigma\in\overline{\rm UnSh}(n-s,s)\\ n=\sigma \alpha(n) \\ \alpha(n)<n \\ (n-s+1)\neq \sigma(n)}} (-1)^{s}[x_{\sigma\alpha\cdot \lambda}] \\
 &\hspace{2cm} + \frac{1}{|{\rm Stab}(\lambda)|n}\frac{(p-1)!}{1!^{m_1}\cdots  p!^{m_p}} \sum_{\substack{\alpha \in \mathcal{S}_n \\ \alpha(n)=n}} C^{-1}_{\alpha\cdot\lambda} \,\,[x_{\alpha\cdot\lambda}]\\
 &= \frac{1}{|{\rm Stab}(\lambda)|n}\frac{(p-1)!}{1!^{m_1}\cdots  p!^{m_p}}\sum_{\alpha \in \mathcal{S}_n} C^{-1}_{\alpha \cdot \lambda }\sum_{\substack{n-1 \geq s\geq 1\\  
 \sigma\in\overline{\rm Sh}(n-s,s)\\ n=\sigma^{-1} \alpha(n)  \\ \alpha(n)<n \\ \sigma(n-s+1)\neq n}} (-1)^{s}\,\,[x_{\sigma^{-1}\alpha\cdot \lambda}] \\
  &\hspace{2cm} + \frac{1}{|{\rm Stab}(\lambda)|n}\frac{(p-1)!}{1!^{m_1}\cdots  p!^{m_p}} \sum_{\substack{\alpha \in \mathcal{S}_n \\ \alpha(n)=n}} C^{-1}_{\alpha\cdot\lambda} \,\,[x_{\alpha\cdot\lambda}]
  \end{align}
  We make the change of variables $(\sigma,\alpha) \mapsto (\sigma,\beta) = (\sigma, \sigma^{-1}\alpha)$ to obtain
  \begin{align}
  &P_{\lambda}(x)= \frac{1}{|{\rm Stab}(\lambda)|n}\frac{(p-1)!}{1!^{m_1}\cdots  p!^{m_p}}\sum_{s=1\ldots n-1}\sum_{\substack{\beta\in \mathcal{S}_n \\ \beta(n)=n}}\sum_{\substack{\sigma \in \overline{{\rm Sh}}(n-s,s)  \\ \sigma(n)<n \\ \sigma(n-s+1)\neq n}} (-1)^s C^{-1}_{ \sigma\beta \cdot \lambda } \,\, [x_{\beta\cdot \lambda}] \\
  &\hspace{2cm} + \frac{1}{|{\rm Stab}(\lambda)|n}\frac{(p-1)!}{1!^{m_1}\cdots  p!^{m_p}} \sum_{\substack{\alpha \in \mathcal{S}_n \\ \alpha(n)=n }} C^{-1}_{\alpha\cdot\lambda} \,\,[x_{\alpha\cdot\lambda}]
\end{align}
We use Lemma \ref{lem:friedrich} to infer
\begin{align}
  P_{\lambda}(x)&= \frac{1}{|{\rm Stab}(\lambda)|n}\frac{(p-1)!}{1!^{m_1}\cdots  p!^{m_p}}\sum_{s=1\ldots n-1}\sum_{\substack{\beta\in \mathcal{S}_n \\ \beta(n)=n}}\sum_{\substack{\sigma \in \overline{{\rm Sh}}(n-s,s) \\\sigma(n-s+1)=n}} (-1)^{s+1} C^{-1}_{ \sigma\beta \cdot \lambda } \,\, [x_{\beta\cdot \lambda}] \\
  &\hspace{2cm} + \frac{1}{|{\rm Stab}(\lambda)|n}\frac{(p-1)!}{1!^{m_1}\cdots  p!^{m_p}} \sum_{\substack{\alpha \in \mathcal{S}_n \\ \alpha(n)=n}} C^{-1}_{\alpha\cdot\lambda} \,\,[x_{\alpha\cdot\lambda}]
  \end{align}
  Set $K = \max\{ 1 \leq q \leq p \colon m_q \neq 0\}$. For any  permutation $\beta \in \mathcal{S}_n$ with $\beta(n)=n$, there exists a couple $(s,w)$ with $s(n) = w(n)=n$, $s\in {\rm Sh}(\lambda)$, $w \in {\rm Stab}(\lambda)$ and $sw=\beta$. Since $|\{ w  \in {\rm Stab}(\lambda) \colon w(n)=n\}| = |{\rm Stab}(\lambda)|/m_K$, we obtain:
  \begin{align}
  &P_{\lambda}(x)= \frac{1}{m_Kn}\frac{(p-1)!}{1!^{m_1}\cdots  p!^{m_p}}\sum_{\substack{\beta\in {\rm Sh}(\lambda)^{+}}}\sum_{s=1\ldots n-1}\sum_{\substack{\sigma \in \overline{{\rm Sh}}(n-s,s) \\\sigma(n-s+1)=n}} (-1)^{s+1} C^{-1}_{ \sigma\beta \cdot \lambda } \,\, [x_{\beta\cdot \lambda}] \label{eqn:lastone}\\
  &\hspace{2cm} + \frac{1}{m_K n}\frac{(p-1)!}{1!^{m_1}\cdots  p!^{m_p}} \sum_{\substack{\alpha \in {\rm Sh}(\lambda)^{+}}} C^{-1}_{\alpha\cdot\lambda} \,\,[x_{\alpha\cdot\lambda}].\label{eqn:lastoneone}
\end{align}
Finally, we have proved the following proposition:
\begin{proposition} With the notations introduced so far, let $\lambda$ be an integer partition of $p$ and $c \in \fg$
\begin{align}
&P_{\lambda}(x)= \frac{1}{m_K n}\frac{(p-1)!}{1!^{m_1}\cdots  p!^{m_p}}\sum_{\substack{\beta\in {\rm Sh}(\lambda)^{+}}}C^{-1}_{\beta\cdot\lambda}+\sum_{s=1\ldots n-1}\sum_{\substack{\sigma \in \overline{{\rm Sh}}(n-s,s) \\\sigma(n-s+1)=n}} (-1)^{s+1} C^{-1}_{ \sigma\beta \cdot \lambda } \quad [x_{\beta\cdot \lambda}].
\end{align}
\end{proposition}
We arrive at our second main result, which is a corollary of the previous proposition. Given a tuple $\ell=(\ell_1,\ldots,\ell_n)$ of positive integers and an ordered partition $(S_1,\ldots,S_q)$, $S_r = \{s^{r}_{1} < \cdots < s^{r}_{k_r}\}$ of $[n]$, we denote by $\ell_{S_1\cdots S_q}$ the tuple
$$
\ell_{S_1\cdots S_q} : = (\ell_{s^1_{1}},\ldots,\ell_{s^1_{k_1}},\ldots,\ell_{s^q_{1}},\ldots,\ell_{s^q_{k_q}})
$$
We use the shorthand notation $\ell ! = \ell_1!\cdots \ell_n!$.
\begin{theorem} With the notations introduced before, the following formula holds, for any $x\in \fg$:
\label{thm:expansion}
\begin{align*}
x^{[p]_{\btr}} &= x^{[p]} + x^{\bullet_{\btr p}}+\sum_{n=2}^{p}\sum_{\substack{\ell=(\ell_1,\ldots,\ell_n) \\ \ell_n \geq 2 \\ \ell_n\geq \ell_i,~ \ell_n > \ell_{n-1} \\ \sum_i \ell_i = p}}\frac{1}{|\{j:\ell_j=\ell_n\}|n}\frac{(p-1)!} {\ell!} \Big(C^{-1}_{\ell} + \sum_{\substack{\emptyset \neq S\subsetneq [n] \\ n \in S }} (-1)^{|S|+1}C^{-1}_{\ell_{S^c}\overline{\ell_{S}}}\Big) \,\, [x_{\ell}]. \nonumber
\end{align*}
\end{theorem}
\begin{example}
$\lambda = 01$ with $p=3$, we get that \eqref{eqn:lastone} is equal to $1$ which sum up to $2=-1$ when adding the contribution \eqref{eqn:lastoneone}.
\end{example}

\section{(Trivially) restricted Post-Lie algebras and trivially restricted $\mathcal{O}$-operators}
In this section, we introduce the notions of \textit{trivially restricted} and \textit{restricted} post-Lie algebras.
\subsection{Trivially restricted post-Lie algebras} 
\begin{definition}\label{def:triviallyrestricted} A post-Lie algebra $\bigl(\fg, [-,-], \btr\bigl)$ over a restricted Lie algebra $\bigl(\fg, [-,-], (-)^{[p]}\bigl)$ is called a \emph{trivially restricted post-Lie algebra} if
\begin{equation}
\label{eqn:trrstcaxone}
x^{[p]_{\btr}} \btr y= (p \text{ times})\, x \btr ( \cdots (x \btr (x \btr y))\cdots ), \quad \forall x,y \in \fg
\end{equation}
and $x\mapsto y \btr x$ is a restricted derivation for all $y\in \fg$, that is,
\begin{align}
\label{eqn:restrictedderoriginalpmap}
y \btr x^{[p]} = (p-1~{\rm times} )\,\,[x, [x, \ldots[x, y \btr x]\cdots],\quad \forall x,y \in \fg.
\end{align}
where the map $(-)^{[p]_{\btr}} $ in Equation \eqref{eqn:trrstcaxone}, is defined by Equation \eqref{eqn:pstruct}.
\end{definition}
%\red  $ y\btr (\cdot)$ restricted derivation de $(\fg, [-,-], (-)^{[p]})$ ça devrait être $ y \btr x^{[p]} = \ad^{p-1}_x(y\btr x).$ (A moins qu'il manque un $\btr$ dans la p-map de (12) )
%\black
\begin{proposition}
\label{prop:fromtriviallyrestrictedtolierestricted}
Any trivially restricted Post-Lie algebra yields a sub-adjacent restricted Lie algebra $(\fg, \llbracket-,-\rrbracket, (-)^{[p]_{\btr}})$. In addition, $\btr$ yields a restricted representation of the restricted sub-adjacent Lie algebra $(\fg,\llbracket-,-\rrbracket, (-)^{[p]_{\btr}})$.
\end{proposition}
\begin{proof}
With the notations of Section \ref{sec:jacobson}, for any $x,y \in \fg$, we have that $(\fg,\llbracket-,-\rrbracket, (-)^{[p]_{\btr}})$ is trivially restricted if and only if for any $x,y \in \fg$
$$
x^{\star_{\btr}p} \btr y = x^{[p]_{\btr}} \btr y \quad (\Leftrightarrow (x^{p}-x^{[p]}) \btr y =0).
$$
Hence, for any $x,y \in \fg$, we have
$$
(p~{\rm times} )\llbracket x, \llbracket x, \ldots,\llbracket x,y \rrbracket \cdots \rrbracket = \llbracket x^{\star_{\btr} p},y \rrbracket = [x^{\star_{\btr} p},y] + x^{[p]_{\btr}}\btr y - y \btr x^{\star_{\btr} p}.
$$
Since $x^{\star_{\btr}p} - x^{[p]_{\btr}} = x^{p} - x^{[p]}$ is a central element in $\mathfrak{u}(\fg)$, we have $[x^{\star_{\btr}p} - x^{[p]_{\btr}},y] = 0$.
From Equation \eqref{eqn:restrictedderoriginalpmap}, we also get 
$$
y \btr (x^{\star_{\btr p}} - x^{[p]_{\btr}})= y \btr (x^{p}-x^{[p]}) = \ad_{x}^{p-1}(x\btr y) - y \btr x^{[p]} = 0,\quad \forall x,y\in \fg.
$$
Hence, 
$(p~{\rm times} )\llbracket x, \llbracket x, \ldots,\llbracket x, - \rrbracket \cdots \rrbracket$
is an inner derivation. Proposition \ref{prop:jacobsonidentitties} ends the proof.
\end{proof}
\begin{proposition} Let $(\fg, \btr, [-,-], (-)^{[p]})$ be a trivially restricted post-Lie algebra. Then, the operation $\btr\colon \mathfrak{u}(\fg)\otimes \mathfrak{u}(\fg) \to \mathfrak{u}(\fg)$ descends to a bilinear operation on the restricted envelope $\mathfrak{u}_p(\fg)$.
\end{proposition}
\begin{proof} 
We have to show that whenever $E\in I$, then for all $X\in \mathfrak{u}(\fg)$, $E \btr X,\, X\btr E \in I$. We have
$$
(x^{p} - x^{[p]}) \btr x_2\cdots x_n = \sum_{i=1}^{n}  x_2\cdots x_{i-1}((x^{p} - x^{[p]}) \btr x_i) x_{i+1}\cdots x_n = 0 \in I
$$
by definition of a restricted post-Lie algebra. Then,
\begin{align}
(x^{p} - x^{[p]})E\btr x&= (x^{p} - x^{[p]}) \btr (E\btr x) - ((x^{p} - x^{[p]}) \btr E) \btr x \\
&=0.
\end{align}
This readily implies, 
\begin{align*}
&(x^{p} - x^{[p]})E\btr x_2\cdots x_n \\
&\hspace{0.25cm}= \sum_{i=1}^n (E_{(1)}\btr x_2)\cdots ((E_{(i-1)}\btr x_{i-1}))((x^{p} - x^{[p]})E_{(i)} \btr x_i) (E_{(i+1)}\btr x_{i+1})\cdots (E_{(n)} \btr x_n) \\
&\hspace{0.25cm}=0.
\end{align*}
Moreover,
\begin{align}
&(y_1\cdots y_m)(x^{p} - x^{[p]})\btr x_2 \\ 
&= y_1 \btr ((y_2\cdots y_m)(x^{p} - x^{[p]})\btr x_2) - (y_1 \btr ((y_2\cdots y_m)(x^{p} - x^{[p]})))\btr x_2.
\end{align}
Since $y\btr (-)$ is a restricted derivation, we get, 
\begin{align}
&(y_1\cdots y_m)(x^{p} - x^{[p]})\btr x_2 = y_1 \btr ((y_2\cdots y_m)(x^{p} - x^{[p]})\btr x_2). 
\end{align}
From the induction hypothesis, we get 
$$
(y_1\cdots y_m)(x^{p} - x^{[p]})\btr x_2 = 0.
$$
Furthermore, 
$$
X\btr E_1(x^{p} - x^{[p]})E_2 = (X_{(1)}\btr E_1)(X_{(2)}\btr (x^{p} - x^{[p]}) )(X_{(3)}\btr E_2),
$$
thus $X\btr E_1(x^{p} - x^{[p]})E_2$ belongs to $I$ if $X_{(2)}\btr (x^{p} - x^{[p]})$ belongs to $I$. We can prove recursively on $n\geq 1$ that $x_2 \cdots x_n \btr(x^{p} - x^{[p]}) $ belongs to $I$. For $n=1$, since $x_2 \btr (-)$ is a restricted derivation, we get that

$$ 
x_2 \btr x^{p} = ad^{p-1}_x(x_2\btr x) = x_2 \btr x^{[p]}.
$$
Let us suppose that the assertion holds for $n \leq N$: 
$$
x_2\cdots x_n \btr (x^p-x^{[p]}) = 0, \quad  n\leq N.
$$
$$
x_1x_2\cdots x_{N} \btr (x^{p} - x^{[p]})= x_1 \btr (x_2\cdots x_{N} \btr (x^{p}-x^{[p]})) - (x_1 \btr( x_2\cdots x_{N})) \btr (x^p-x^{[p]})
= 0,
$$
where we have used the induction hypothesis to write the last equality.
\end{proof}
\begin{corollary} Consider the following associative product on $\mathfrak{u}_p(\fg)$ : 
$$
E \star_{\btr} F = E_{(1)}(E_{(2)} \btr F),\quad E,F \in \mathfrak{u}_p(\fg).
$$
Then, the map
$$
\mathfrak{u}_p(\fg,\llbracket -,-\rrbracket) \to \mathfrak{u}_p(\fg),\quad x_2\cdots x_n \mapsto x_2 \star_{\btr} \cdots \star_{\btr} x_n
$$
yields an algebra isomorphism between the restricted envelope of the sub-adjacent Lie bracket of the restricted post-Lie algebra $(\fg,\btr, (-)^{[p]_{\btr}} [-,-],(-)^{[p]})$ to the algebra $(\mathfrak{u}_p(\fg), \star_{\btr})$.
\end{corollary}
\subsection{Trivially restricted $\mo$-operators}
\begin{definition}
Let $(\fg,\btr,\fa)$ be a restricted Lie module over $k$ and let ${\theta}\colon \fg \to \fa$ be an ordinary $\mathcal{O}$-operator. We will say that ${\theta}$ is \emph{trivially restricted} if for any $x \in \fg$, we have
\begin{align}
{\theta}(x^{[p]_{\btr_{\theta}}}) = \theta(x)^{[p]},
\end{align} where $(-)^{[p]_{\btr_{\theta}}}$ is given by Equation \eqref{eqn:pstruct} for the post-Lie product given by Equation \eqref{eq:postlie-from-oop}.
\end{definition}

\begin{proposition}\label{prop:restricted-o-op}
Let $(\fg,\btr,\fa)$ be a restricted Lie module over $k$ and let $\theta\colon \fg \to \fa$ be a trivially restricted $\mathcal{O}$-operator. Then, the post-Lie algebra $(\fg,[-,-],\btr_{\theta})$ is trivially restricted, where $x\btr_{\btr_{\theta}} y:={\theta}(x)\btr y,\quad\forall x,y\in \fg$.
\end{proposition}
\begin{proof}
Let $x,y\in\fg$. We have
\begin{align*}
x^{[p]_{\theta}}\btr_{\theta} y={\theta}(x^{[p]_{\theta}})\btr y={\theta}(x)^{[p]}\btr y&={\theta}(x)\btr\bigl(\cdots\btr ({\theta}(x)\btr y)\cdots\bigl)\\&=x\btr_{\theta}\bigl(\cdots\btr_{\theta} (x\btr y)\cdots\bigl).
\end{align*} Moreover, we have
\begin{align*}
 y\btr_{\theta} x^{[p]}={\theta}(y)\btr x^{[p]}=\ad_x^{p-1}\bigl({\theta}(y)\btr x\bigl)=\ad_x^{p-1}(y\btr_{\theta} x).
\end{align*}
\end{proof}

\begin{example} We continue Example \ref{sec:braces} and prove that
$
(\mathbb{C}[\mathcal{P}],\btr_{\rho}, (-)^{p})
$ is a trivially restricted post-Lie algebra, where $(-)^{p}$ is the $p^{th}$ power respectively to the product $m$.
We first check \eqref{eqn:trrstcaxone}, that is, 
$$
{\bm \rho}(x^{\cdot p} - x^{p}) = 0, \quad \forall x\in\fg.
$$
An explicit formula for $\bm{\rho}\colon \mathfrak{u}(\mathbb{C}[\mathcal{P}])\to\mathfrak{u}(\mathbb{C}[\mathcal{P}]_0)$ is given by
$$
{\bm \rho}(x_2\cdots x_n) = \sum_{(B_1,\ldots,B_k) \in {\rm Part}(n)} \rho_{}(x_{B_1})\cdots \rho_{}(x_{B_k}),
$$
where $\rho(x_{B})=Ix_{b_1}\cdots x_{b_s}$ if $B_1 = \{b_1 < \cdots < b_s\}$
and ${\rm Part}(n)$ is the set of ordered partitions of $[n]$. Since 
$$
{\bm \rho}(x^{\cdot p}) = \sum_{k=1}^{p}\sum_{q_1+\ldots+q_k=p, q_j\geq 1}\binom{p}{q_1,\ldots,q_k}\sum_{\sigma \in S_k} \rho(x^{ q_{\sigma^{-1}(1)}})\cdots \rho(x^{ q_{\sigma^{-1}(k)}}) = \rho(x^p),
$$
we get that $\rho$ is trivially restricted and thus from Proposition \ref{prop:restricted-o-op} that $(\mathbb{C}[\mathcal{P}],\btr_{\rho}, [-,-])$ is trivially restricted post-Lie algebra and ${\bm \rho}$ is a restricted morphism.
\end{example}

\subsection{Restricted post-Lie algebras}
\begin{definition}
\label{def:restrictedpostlie}
We say that a tuple $(\fg, \btr ,  (-)^{[p]_{\btr}} ,  [-,-], (-)^{[p]})$ is a \emph{restricted post-Lie algebra} if 
\begin{enumerate}

\item \label{item:rstpostlieone}$(\fg, \btr, [-,-])$ is a post-Lie algebra;

\item \label{item:rstpostlietwo}$(\fg, [-,-], (-)^{[p]})$ is a restricted Lie algebra;

\item \label{item:rstpostliezero}$(-)^{[p]_{\btr}}$ is $p$-linear and satisfies the Jacobson identities \eqref{def:restrictedLie} with respect to $\llbracket-,-\rrbracket$,

\item \label{item:rstpostliefive} $
[x^{[p]_{\btr}},y] + x^{[p]_{\btr}}\btr y  =  y \btr x^{[p]_{\btr}} + (p~{\rm times} )\llbracket x, \llbracket x, \ldots,\llbracket x,y \rrbracket \cdots \rrbracket, \quad \forall x,y \in \fg
$,
\item \label{item:rstpostliethree} for any $x\in\fg$, $x\btr$ is a restricted derivation on  $([-,-], (-)^{[p]})$;

\item \label{item:rstpostliefour}$ x^{[p]_{\btr}} \btr y= x \btr ( x \btr ( \cdots (x \btr y)\cdots ), \quad \forall x,y \in \fg $;

\end{enumerate}
where $\llbracket-,-\rrbracket$ refers to the sub-adjacent Lie bracket, see Equation \eqref{eq:subadjacent-bracket}.
\end{definition}
\begin{remark} In Definition \ref{def:restrictedpostlie}, $(-)^{[p]_{\btr}}$ is not assumed to satisfy Equation \eqref{eqn:pstruct}. We can replace items \ref{item:rstpostliezero} and \ref{item:rstpostliefive} by $(\mathfrak{g}, \llbracket -,- \rrbracket, x^{[p]_{\btr}})$ is restricted.
\end{remark}

\begin{proposition}
A trivially restricted post-Lie algebra $(\fg,\btr, [-,-],(-)^{[p]})$ is a restricted post-Lie algebra in the sense of Definition \ref{def:restrictedpostlie}.
\end{proposition}
\begin{proof}
Let  $(\fg,\btr, [-,-], (-)^{[p]})$ be a trivially restricted post-Lie algebra. In that case, the map $(-)^{[p]_{\btr}}$ is given by Equation \eqref{eqn:pstruct}. Points \ref{item:rstpostlieone}, \ref{item:rstpostlietwo}, \ref{item:rstpostliezero}, \ref{item:rstpostliethree} and \ref{item:rstpostliefour} hold by definition. Thus, from Proposition \ref{prop:fromtriviallyrestrictedtolierestricted}, $(\fg,\llbracket -,- \rrbracket, (-)^{[p]_{\btr}})$ is a restricted Lie algebra, which implies that for any $x,y\in \fg$,
$$
\llbracket x^{[p]_{\btr}}, y\rrbracket = (p~{\rm times} )\,\llbracket x, \llbracket x, \ldots,\llbracket x,y \rrbracket \cdots \rrbracket.
$$
Hence, by inserting the definition of the sub-adjacent Lie bracket, we have
$$
[x^{[p]_{\btr}},y] + x^{[p]_{\btr}}\btr y  =  y \btr x^{[p]_{\btr}} + (p~{\rm times} )\,\llbracket x, \llbracket x, \ldots,\llbracket x,y \rrbracket \cdots \rrbracket,~\forall x,y\in\fg.
$$
\end{proof}

\begin{proposition} 
\label{prop:restrictedadjacent}
Let $(\fg,\btr, [-,-], (-)^{[p]})$ be a restricted post-Lie algebra over $k$. Then the triple $(\fg, \llbracket -,- \rrbracket, (-)^{[p]_{\btr}})$ is a restricted Lie algebra over $k$.
\end{proposition}
\begin{proof} The proposition follows from a straightforward check of the axioms listed in Definition \ref{def:restrictedpostlie}. From Proposition \ref{prop:jacobsonidentitties}, the Jacobson identities hold. We have to show that $(p~{\rm times} )\llbracket x, \llbracket x, \ldots,\llbracket x,- \rrbracket \cdots \rrbracket$ is an inner derivation for all $x\in \fg$. Let $y\in \fg$, from \eqref{item:rstpostliefive}, we get 
\begin{align*}
\llbracket x, \llbracket x, \ldots,\llbracket x,y \rrbracket \cdots \rrbracket = [x^{[p]_{\btr}},y] + x^{[p]_{\btr}}\btr y - y \btr x^{[p]_{\btr}} = \llbracket x^{[p]_{\btr}}, y\rrbracket,\quad\forall x,y\in \fg,
\end{align*}
which concludes the proof.
\end{proof}
% \red peut-être que l'on devrait formuler le résultat de la façon suivante.

% \begin{Proposition}
%     Let $(\fg, \btr, [-,-], (-)^{[p]})$ be a post-Lie algebra over a restricted Lie algebra $(\fg, [-,-], (-)^{[p]})$. Then, the adjacent Lie algebra is restricted with respect to the p-map $(-)^{[p]_{\btr}}$ only if Eq. 2. 3. 4. above hold. In that case, $\btr$ defines a restricted representation of the adjacent algebra.
% \end{Proposition}

% \hl{j'ai ajouté une proposition plus haut}
% \black

\begin{example}
    Let $\text{char}(k)=3$. Let $\fg$ be the two dimensional restricted Lie algebra spanned by $e_1,e_2$ with the only nonzero bracket given by $[e_1,e_2]=e_2$ and the 3-map $e_1^{[3]}=e_1,~e_2^{[3]}=0.$ Then, the four following products define a trivially restricted post-Lie algebra structure on $\fg$:
    \begin{align}
        &1.~e_1\btr e_1=0,&e_1\btr e_2=-e_2,&&e_2\btr e_1=0,&& e_2\btr e_2=\lambda e_2;\quad && \lambda\in k, \\
        &2.~e_1\btr e_1=\lambda e_2,&e_1\btr e_2= e_2,&&e_2\btr e_1=0,&&e_2\btr e_2=0;\quad && \lambda\in k, \\
        &3.~e_1\btr e_1=\lambda e_2,&e_1\btr e_2=e_2,&&e_2\btr e_1=0,&&e_2\btr e_2=0;\quad && \lambda\in k, \\
        &4.~e_1\btr e_1=\lambda e_2,&e_1\btr e_2= 0,&&e_2\btr e_1=e_2,&&e_2\btr e_2=0;\quad && \lambda\in k,
    \end{align}
\end{example}

\subsubsection{The case $p=2$.}\label{sec:p=2}

In this section, we work over a field $k$ of characteristic $p=2$.
We specialize the definition \ref{def:triviallyrestricted} of a trivially restricted post-Lie algebra to $p=2$.
\begin{definition}\label{def:restrictedpostlie2}
A trivially restricted post-Lie algebra is a tuple $(\fg,[-,-],(-)^{[2]},\btr)$, where
\begin{enumerate}
    \item $(\fg,[-,-],(-)^{[2]})$ is a restricted Lie algebra;
    \item $(\fg,[-,-],\btr)$ is an ordinary post-Lie algebra;
    \item the following conditions are satisfied, for all $x,y\in \fg$:
    \begin{align}
        x\btr y^{[2]}&=[x\btr y,y];\label{eq:restderivation2}\\
        x^{[2]}\btr y&=x\btr(x\btr y)+(x\btr x)\btr y.\label{eq:associator2}
    \end{align}
\end{enumerate}
\end{definition}
In this case, the sub-adjacent $2$-structure, obtained by specializing \eqref{eqn:pstruct}, reads
\begin{equation}
x^{[2]_{\btr}}=x^{[2]}+ x\btr x, \quad \forall x\in \fg.
\end{equation}
\begin{proposition}[restricted version of \text{\cite[Theorem 3.1]{JZ}}, $p=2$]\label{prop:JZp=2}
      Let $\A$ be a commutative associative algebra. Let $\D\subset\Der(\A)$ be a restricted subalgebra. Then, the space $\A\otimes\D$ is a restricted post-Lie algebra with the bracket and the 2-map defined in Proposition \ref{prop:tensorderivation} and the product
      \begin{equation}
        (a\otimes f)\btr(b\otimes g):=af(b)\otimes g,~\forall a,b\in \A,~\forall f,g\in \D.     
        \end{equation}
      \end{proposition}

      \begin{proof}
        We prove the identities involving the 2-map. Let $a,b\in \A$ and $f,g\in \D$.  Condition \eqref{eq:restderivation2} of Definition \ref{def:restrictedpostlie2} follows from Equation \ref{eq:JDJderivation}. Furthermore, we have
        \begin{align*}
            (a\otimes f)^{[2]}\btr (b\otimes g)=(a^2\otimes f^2)\btr(b\otimes g)=a^2f\circ f(b)\otimes g;
        \end{align*} and
        \begin{align*}
            (a\otimes f)\btr\bigl((a\otimes f)\btr(b\otimes g)\bigl)+\bigl((a\otimes f)&\btr(a\otimes f)\bigl)\btr(b\otimes g)\\
            =~&(a\otimes f)\btr \bigl( af(b)\otimes g \bigl)+\bigl(af(a)\otimes f\bigl)\btr(b\otimes g)\\
            =~&af\bigl(af(b)\bigl)\otimes g+af(a)f(b)\otimes g\\
            =~&af(a)f(b)\otimes g+a^2f\circ f(b)\otimes g+af(a)f(b)\otimes g\\
            =~&a^2f\circ f(b)\otimes g.
        \end{align*} Thus, Condition \eqref{eq:associator2} of Definition \ref{def:restrictedpostlie2} is satisfied as well and $\btr$ defines a trivially restricted post-Lie structure.
 \end{proof}
 
\begin{example} Let $\fg=\langle e_1,e_2,e_3\rangle$, with the only non-trivial bracket $[e_2,e_3]=e_2$ and 2-map $e_3^{[2]}=e_3, e_2^{[2]} = 0, e_1^{[2]} = 0$. Let $\lambda,\mu,\gamma\in k$. The $2$-map can be computed on arbitrary elements by the formula
            $$(\lambda e_1+\mu e_2+\gamma e_3)^{[2]}=\lambda^2e_1^{[2]}+\mu^2e_2^{[2]}+\gamma^2e_3^{[2]}+[\lambda e_1,\mu e_2+\gamma e_3]+[\mu e_2,\gamma e_3]=\gamma^2e_3+\mu\gamma e_2.$$         
Then, there are two families of trivially restricted post-Lie products on $\fg$ given by
            \begin{align*}
            \btr_1:&~e_3\btr_1e_2=\alpha e_2,~e_3\btr_1e_3=\beta e_2,~ \alpha,\beta\in k;\\
             \btr_2:&~e_1\btr_2e_3=\xi e_2,~\xi\in k.
            \end{align*}
        Indeed, the restricted derivations of $\fg$ are given by the restricted 1-cocycles with adjoint coefficients (see \cite{EF}), therefore we have
        \begin{equation}
            \Der_{\rm res}(\fg)=\Span\{e_1\otimes e_1^*,e_2\otimes e_2^*,e_2\otimes e_3^*\}.
        \end{equation}
        Thus, a restricted post-Lie product $\btr$ on $\fg$ is given, for scalars $\lambda_i,\mu_i,\gamma_i\in k$, by
        \begin{align*}
            e_1\btr(-)&=\lambda_1e_1\otimes e_1^*+\mu_1e_2\otimes e_2^*+\gamma_1e_2\otimes e_3^*;\\
            e_2\btr(-)&=\lambda_2e_1\otimes e_1^*+\mu_2e_2\otimes e_2^*+\gamma_2e_2\otimes e_3^*;\\
            e_3\btr(-)&=\lambda_3e_1\otimes e_1^*+\mu_3e_2\otimes e_2^*+\gamma_3e_2\otimes e_3^*.
        \end{align*}
        Checking the remaining conditions on the basis $e_1,e_2,e_3$ leads to the conclusion.
\end{example}

Over a field of characteristic 2, the notions of pre-Lie and post-Lie \textit{super}algebras have been investigated in \cite{BBE}.

 \subsubsection{The case $p=3$.}\label{sec:p=3}
In this section, we work over a field $k$ of characteristic $p=3$ and specialize the definition \ref{def:triviallyrestricted}.
\begin{definition}\label{def:restrictedpostlie3} 
A trivially restricted post-Lie algebra is a tuple $(\fg,[-,-],(-)^{[3]},\btr)$ where
\begin{enumerate}
    \item $(\fg,[-,-],(-)^{[3]})$ is a restricted Lie algebra;
    \item $(\fg,[-,-],\btr)$ is an ordinary post-Lie algebra;
    \item the following conditions are satisfied, for all $x,y\in \fg$:
    \begin{align}
        x\btr y^{[3]}&=\ad_y^2(x\btr y);\label{eq:restderivation3}\\
        x^{[3]}\btr y&=x\btr(x\btr(x\btr y))+2(x\btr x)\btr(x\btr y)+x\btr((x\btr x)\btr y)\label{eq:associator3}\\\nonumber
        &\quad+((x\btr x)\btr x)\btr y+(x\btr(x\btr x) )\btr y.
    \end{align}
\end{enumerate}
\end{definition}
%%%%%%%%%%%%%%%%%%%%%%%%%%%
In this case, the sub-adjacent $3$-structure, obtained by specializing \eqref{eqn:pstruct}, reads
\begin{align}
        x^{[3]_{\btr}}&:=x^{[3]}+x\btr(x\btr x)+[x\btr x,x],~\forall x\in \fg.
    \end{align} 

The compatibility condition \eqref{eq:associator3} is equivalent to the fact that $\btr$ yields a restricted representation of $\bigl(\fg,[-,-],(-)^{[3]_{\btr}}\bigl)$. Indeed, from Eq. \eqref{eqn:trrstcaxone}, we have for all $x,y\in\fg,$
$$x^{[3]_{\btr}}\btr y=x\btr (x\btr (x\btr y)), $$ which is equivalent to
\begin{align*}
    x^{[3]}\btr y&=-(x\btr(x\btr x))\btr y-[x\btr x,x]\btr y+x\btr(x\btr(x\btr y))\\
    &=-(x\btr(x\btr x))\btr y+x\btr(x\btr(x\btr y))-(x\btr x)\btr(x\btr y)\\
    &\quad+((x\btr x)\btr x)\btr y+x\btr((x\btr x)\btr y)-(x\btr(x\btr x))\btr y\\
    &=x\btr(x\btr(x\btr y))+2(x\btr x)\btr(x\btr y)+x\btr((x\btr x)\btr y)\\
        &\quad+((x\btr x)\btr x)\btr y+(x\btr(x\btr x) )\btr y.
\end{align*}

\begin{example}
   Let $\text{char}(k)=3.$ Let $\fg$ be the restricted Heisenberg algebra spanned by elements $\{e_1,e_2,e_3\}$ with the only nonzero bracket $[e_1,e_2]=e_3$ and $3$-map $e_1^{[3]}=e_3$ (see \cite{EM}). Let $\mu,\gamma,\theta\in k$. Then, the product $\btr$ defined by the nonzero relations
    \begin{equation}
        e_1\btr e_1=\mu e_3,~e_1\btr e_2=\gamma e_3,~\text{and }e_2\btr e_1=\theta e_3
    \end{equation} yields a trivially restricted post-Lie algebra in the sense of Definition \ref{def:restrictedpostlie}.     
\end{example}

\begin{example} (adapted from \cite{PLBG}) Let $k$ be a field of characteristic 3 containing $\sqrt{2}$. We consider the Lie algebra $\fg=\mathfrak{sl}_2(k)$ spanned by elements $e_1,e_2,e_3$ with brackets $[e_2,e_3]=e_1,~[e_3,e_1]=e_2,~[e_1,e_2]=e_3$ and the 3-map $e_i^{[3]}=-e_i,~\forall i=1,2,3.$ Denote by $\alpha=1+\sqrt{2}$ and by $\beta=2\sqrt{2}+1$. Then, the product $\btr$ defined as follows endows $\fg$ with a trivially restricted post-Lie structure:
\begin{align*}
    e_1\btr e_i&=-[e_1,e_i];\\
    e_2\btr e_i&=-[\alpha e_2+\beta e_3,e_i];\\
    e_3\btr e_i&=-[\alpha e_2+\beta e_3,e_i].
\end{align*}

\end{example}

\begin{proposition}[restricted version of \text{\cite[Theorem 3.1]{JZ}}, $p=3$]\label{prop:JZp=3}
      Let $\A$ be a commutative associative algebra. Let $\D\subset\Der(\A)$ be a restricted subalgebra. Then, the space $\A\otimes\D$ is a trivially restricted post-Lie algebra with the bracket and the 3-map defined in Proposition \ref{prop:tensorderivation} and the post-Lie product
      \begin{equation}
        (a\otimes f)\btr(b\otimes g):=af(b)\otimes g,~\forall a,b\in \A,~\forall f,g\in \D.     
        \end{equation}
      \end{proposition}

      \begin{proof}
Condition \eqref{eq:restderivation3} of Definition \ref{def:restrictedpostlie3} follows from Equation \eqref{eq:JDJderivation}. We prove Condition \eqref{eq:associator3}. Let $a,b\in \A $ and $ f,g\in \D.$ We compute separately the five terms of the right-hand side of Condition \eqref{eq:associator3}.

        \begin{align*}
            (a\otimes f)\btr\bigl((a\otimes f)\btr((a\otimes f)\btr(b\otimes g)\bigl)&=af\bigl(af(a)f(b)+a^2f^2(b)\bigl)\otimes g\\
                    &=af(a)^2f(b)\otimes g+af^2(a)f(b)\otimes g+\underline{a^3f^3(b)\otimes g};\\
                    &~\\
            2\bigl((a\otimes f)\btr(a\otimes f)\bigl)\btr\bigl((a\otimes f)\btr(b\otimes g)\bigl)&=  2\bigl(af(a)f(af(b))\otimes g  \bigl)\\
            &=2af(a)^2f(b)\otimes g+2a^2f(a)f^2(b)\otimes g;\\
                    &~\\
            (a\otimes f)\btr\bigl(((a\otimes f)\btr (a\otimes f))\btr (b\otimes g) \bigl)&= af\bigl(af(a)f(b) \bigl)\otimes g\\
            &=af(a)^2f(b)\otimes g+a^2f^2(a)f(b)\otimes g+a^2f(a)f^2(b)\otimes g;\\
                    \end{align*}
                    \begin{align*}
            \bigl(  ((a\otimes f)\btr(a\otimes f))\btr(a\otimes f)  \bigl)\btr(b\otimes g)&=\bigl(af(a)^2\otimes f   \bigl)\btr(b\otimes g)\\
            &=af(a)^2f(b)\otimes g;\\
                    &~\\
            \bigl( (a\otimes f)\btr ((a\otimes f)\btr(a\otimes f) )   \bigl)\btr(b\otimes g)&=af(a f(a))f(b)\otimes g\\
            &= af(a)^2f(b)\otimes g+a^2f^2(a)f(b)\otimes g.
        \end{align*}
        Summing those five terms, everything vanishes except the underlined term. Or, computing the left-hand side of Condition
        \eqref{eq:associator3} gives
        \begin{equation}
            (a\otimes f)^{[3]}\btr((b\otimes g)=(a^3\otimes f^3)\btr(b\otimes g)=a^3f^3(b)\otimes g,
        \end{equation} which finishes the proof.
      \end{proof}
      \subsection{Quasi-shuffle algebras.} A \emph{quasi-shuffle algebra} (see \cite{H}) is a tuple $(A, \prec, \succ, \cdot)$, where $A$ is a vector space, $\cdot$ is an associative product on $A$, the binary operations $\prec$ and $\succ$ satisfy for all $a,b,c\in A$ the relations
      \begin{align}
      &(a \prec b ) \prec c = a \prec (b \star c), &&  a \succ ( b  \succ c ) = (a \star b) \succ c && (a \succ b) \prec c = a \succ (b \prec c) \\
      &(a\cdot b) \prec c = a \cdot (b \prec c)   && (a\succ b) \cdot c = a \succ (b \cdot c) && a\cdot (b \succ c) = (a  \prec b ) \cdot c
      \end{align}
 with the associative product $a\star b := a\prec b+ a\succ b+ a\cdot b,\,\forall a,b\in A$. The above relations implies that $ a \btr b : = b \prec a - a \succ b$ is a post-Lie operation with respect to the Lie algebra $(A, [-,-]_{\cdot})$. Note that quasi-shuffle algebras also appear in the literature under the name dendriform trialgebras (or tridendriform algebras), see e.g. \cite{LR}.

 \begin{proposition} Let $(A, \prec, \succ, \cdot)$ be a quasi-shuffle algebra. Then, 
 $
 (A, \btr, (-)^{\star p},[-,-]_{\cdot}, (-)^{\cdot p})
 $
 is a restricted post-Lie algebra (see Definition \ref{def:restrictedpostlie}).
  \end{proposition}
   \begin{proof} We check points 3., 4. and 5. from the Definition \ref{def:restrictedpostlie}. We check point 4.
   For any $a \in A$, set $R_a\colon A \to A$, $R_a(x) = a \succ x$ and $L_a\colon A \to A$, $L_a(x) = x \prec a$.
   Then 
   $$
   L_a R_b (x) = a\succ (x \prec b) = (a\succ x) \prec b = R_b \circ L_a (x), \quad \forall x \in A.
   $$
   Hence, for any $x,y \in A$,
   \begin{align*}
  (p ~{\rm times}) \quad  x \btr( x \btr( \cdots (x \btr y) \cdots ) & = (L_x - R_x)^{p}(y)  \\
  & = \sum_{k=0}^{p}\binom{p}{k}L_x^k (- R_x)^{p-k}\\  & = L_x^p - R_x^p (y) = L_{x^{\star p}} - R_{x^{\star p}}(y) = x^{\star p} \btr y.
  \end{align*}
We prove 3. For any $x,y\in A$, we have
\begin{align}
z \btr (x\cdot y) &= (x\cdot y) \prec z - z \succ (x\cdot y) \\ &= x \cdot (y \prec z) - (z \succ x) \cdot y \\
&= x \cdot (y \prec z) - x \cdot (z \succ y) + (x \prec z)\cdot y - (z \succ x) \cdot y \\
& = x \cdot (x \btr y) + (z \btr x)\cdot y.
\end{align}
\end{proof}

\noindent\textbf{Aknowledgements.} We would like to thank S. Benayadi, S. Bouarroudj and A. Dzhumadil'daev for their support, their interest in this project and for fruitful discussions, as well as an anonymous referee for helpful and encouraging comments.

\end{document}